\documentclass[11pt,letterpaper,reqno]{amsart}
\usepackage{graphicx}
\usepackage{amsmath}
\usepackage{amssymb}
\usepackage{amsfonts}
\usepackage{amsrefs}
\usepackage{enumerate}
\usepackage{cancel}
\usepackage[mathscr]{eucal}

\usepackage[usenames]{color}

\newcommand{\R}{\mathbb{R}}

\newcommand{\F}{\mathcal{F}}

\newtheorem{thm}{Theorem}
\newtheorem{lemma}[thm]{Lemma}
\newtheorem{prop}[thm]{Proposition}
\newtheorem{cor}[thm]{Corollary}
\newtheorem{conj}[thm]{Conjecture}
\theoremstyle{definition}
\newtheorem{definition}[thm]{Definition}

\theoremstyle{remark}
\newtheorem{remark}{Remark}

\newtheorem*{outline}{Outline}

\DeclareMathOperator{\Hess}{Hess}
\DeclareMathOperator{\capac}{cap}

\begin{document}

\title{ADM mass and the capacity-volume deficit at infinity}
\author{Jeffrey L. Jauregui}
\address{Dept. of Mathematics,
Union College, 807 Union St.,
Schenectady, NY 12308,
United States}
\email{jaureguj@union.edu}
\date{\today}

\begin{abstract}
Based on the isoperimetric inequality, G. Huisken proposed a definition of total mass in general relativity that is equivalent to the ADM mass for (smooth) asymptotically flat 3-manifolds of nonnegative scalar curvature, but that is well-defined in greater generality. In a similar vein, we use the isocapacitary inequality (bounding capacity from below in terms of volume) to suggest a new definition of total mass. We prove an inequality between it and the ADM mass, and prove the reverse inequality with harmonically flat asymptotics, or, with general asymptotics, for exhaustions by balls (as opposed to arbitrary compact sets). This approach to mass may have applications to problems involving low regularity metrics and convergence in general relativity, and may have some advantages relative to the isoperimetric mass. Some conjectures, analogs of known results for CMC surfaces and isoperimetric regions, are proposed.
\end{abstract}

\maketitle 

\section{Introduction}
In general relativity, the ADM mass \cite{ADM} of an asymptotically flat (AF) 3-manifold $(M,g)$ captures the total physical mass of $(M,g)$ viewed as a time-symmetric initial data set for Einstein's equation. It is defined through the formula
\begin{equation}
\label{eqn_adm}
m_{ADM}(M,g) = \lim_{r \to \infty} \frac{1}{16\pi}\int_{S_r} \left(\partial_i g_{ij} - \partial_j g_{ii} \right) \frac{x^j}{|x|} dA,
\end{equation}
where the metric coefficients $g_{ij}$ decay to $\delta_{ij}$ in an appropriate coordinate chart $(x^i)$ near infinity and $S_r$ is the coordinate sphere $|x|=r$. Bartnik \cite{Ba0} and Chru\'sciel \cite{Chr} proved the ADM mass is well-defined, i.e. the limit exists and is independent of the asymptotically flat coordinate chart. When the scalar curvature is nonnegative (physically corresponding to nonnegative energy density) and the boundary $\partial M$ is empty or consists of minimal surfaces (corresponding to apparent horizons of black holes), the ADM mass is known to be nonnegative, equaling zero only for Euclidean space. This foundational result is the positive mass theorem, first proved by Schoen and Yau \cite{SY} and by Witten \cite{W}. 

Several other formulas that produce the ADM mass are known, such as the $r \to \infty$ limit of the Brown--York mass or the Hawking mass of $S_r$; see \cite{FST} and the references therein. In a different spirit, the ADM mass can also be recovered through an expression involving the Ricci curvature at infinity; see \cite{MT} and the references therein. All of these expressions and of course \eqref{eqn_adm} depend on, at least, first derivatives of the Riemannian metric.

Many problems in general relativity involving mass naturally lead one to consider low regularity metrics and/or low regularity convergence. For example, a well-known conjecture is the almost-equality case of the positive mass theorem: a sequence of AF manifolds of nonnegative scalar curvature with ADM masses converging to zero must converge (in some appropriate sense) to Euclidean space. However, convergence of the metric tensors in the local $C^2$ or even $C^0$ sense  is too fine for this to be true. Some recent results such as those by Lee and Sormani \cite{LS}, Huang, Lee, and Sormani \cite{HLS}, and Sormani and Stavrov Allen \cite{SS} have strongly suggested such a statement may hold for pointed Sormani--Wenger intrinsic flat ($\mathcal{F}$) convergence \cite{SW}. Closely related but even more challenging is the direct approach to the Bartnik mass-minimization problem \cites{Ba1, Ba2}, essentially showing that a sequence of AF manifolds as above, but all containing a fixed compact set isometrically, with ADM masses limiting to the smallest possible value, must converge outside that set to a static vacuum metric. Again, this can only hold for a relatively weak notion of convergence, such as $\mathcal{F}$, and it seems unavoidable to consider limits that are not even manifolds but metric spaces, possibly with additional structure. Motivated in part by this Bartnik problem, the author as well as the author and Lee showed lower semicontinuity of the ADM mass for $C^2$, $C^0$ and ultimately volume-preserving $\F$-convergence \cites{J_LSC, JL1,JL2}. We refer the reader to Sormani's survey for further discussion of low regularity convergence problems involving the mass and scalar curvature \cite{Sor}.

\medskip

For low regularity problems involving mass,  \eqref{eqn_adm} and the other formulas mentioned above are problematic due to the dependence on derivatives of the metric coefficients. This was one source of motivation for Huisken's definition of isoperimetric mass \cites{Hui1, Hui2}, which is elegantly defined using only areas and volumes (i.e., coarse quantities). The idea is that far out in the asymptotically flat end, the Euclidean isoperimetric inequality almost holds, and the total mass should be detectable in the deficit.

\begin{definition}[Huisken \cites{Hui1,Hui2}]
The \emph{isoperimetric mass} of an asymptotically flat 3-manifold $(M,g)$ is:
\begin{equation}
\label{miso}
m_{iso}(M,g) = \sup_{\{K_j\}} \limsup_{j \to \infty}  \frac{2}{|\partial^* K_j|} \left[|K_j| - \frac{1}{6\sqrt{\pi}} |\partial^*K_j|^{3/2}\right],
\end{equation}
where the supremum is taken over all exhaustions $\{K_j\}$ of $M$ by compact sets, and $|\partial^*K_j|$ and $|K_j|$ are the perimeter and volume of $K_j$, both with respect to $g$.
\end{definition}
Note that $m_{iso} \in [-\infty, \infty]$  is well-defined even if $g$ is merely a $C^0$ Riemannian metric. (See the appendix of \cite{JL2} for a treatment of perimeter for $C^0$ Riemannian metrics. Recall that if $\partial K$ is smooth, the perimeter is simply the area of $\partial K$.) 

As an example: in Euclidean 3-space, the statement $|K_j| - \frac{1}{6\sqrt{\pi}} |\partial^*K_j|^{3/2} \leq 0$ is precisely the classical isoperimetric inequality, so $m_{iso}(\R^3, \delta_{ij}) \leq 0$. That Euclidean 3-space has $m_{iso}$ equal to zero can then be seen from taking an exhaustion by round balls. 

The most compelling reason for regarding $m_{iso}$ as a ``mass'' is the following theorem, first proposed by Huisken \cites{Hui1, Hui2}. 
\begin{thm}
\label{thm_miso}
Let $(M,g)$ be an asymptotically flat 3-manifold of nonnegative scalar curvature whose boundary is either empty or minimal. Then $m_{iso}(M,g) = m_{ADM}(M,g)$.
\end{thm}
The inequality $m_{iso}(M,g) \geq m_{ADM}(M,g)$ was shown by Miao using results of Fan, Shi, and Tam (see \cite{FST}). The author and Lee proved $m_{iso}(M,g) \leq m_{ADM}(M,g)$ by working with Huisken's idea of combining mean curvature flow and inverse mean curvature flow, along with generalizing a monotone quantity he discovered \cite{JL1}.  A different proof of this second inequality was given by Chodosh, Eichmair, Shi, and Yu by showing essentially that large CMC spheres enclose isoperimetrically optimal regions \cites{CESY}. We note that both proofs of this second inequality utilize the weak inverse mean curvature flow of Huisken and Ilmanen \cite{HI}.

Thus, in nonnegative scalar curvature, the mass can be recovered from the top-order term of the isoperimetric deficit relative to Euclidean space. An exhaustion attains the supremum in the definition of $m_{iso}$ when the compact sets are becoming as isoperimetrically efficient as possible as $j \to \infty$.

\medskip 
The primary purpose of this paper is to carry out an analog of Huisken's program for defining the mass based on the \emph{isocapacitary} inequality (relating capacity to volume) in place of the isoperimetric inequality, and to hint that the ``isocapacitary mass'' might be particularly well-suited to low regularity problems. Recall the \emph{capacity} of a compact set $K \supseteq \partial M$ in an AF 3-manifold $(M,g)$ is the number
\begin{equation}
\label{eqn_cap}
\capac_g(K) = \inf_{\psi} \frac{1}{4\pi} \int_M |\nabla \psi|_g^2 dV_g,
\end{equation}
where the infimum is taken over all locally Lipschitz functions $\psi$ on $M$ that vanish on $K$ and uniformly approach 1 at infinity. We will drop the $g$ subscripts when the choice of metric is clear. From standard considerations in the calculus of variations and elliptic theory, a minimizer in \eqref{eqn_cap} must be harmonic in $M \setminus K$. If $\partial K$ is mildly regular (e.g., a $C^1$ surface), this infimum is uniquely achieved by the capacitary potential of $K$, i.e. by the unique function $\varphi$ that is harmonic on $M \setminus K$, tending to 1 at infinity, vanishing on $\partial K$,  extended by 0 to $K$. Then by the divergence theorem, $\capac(K)$ equals $\frac{1}{4\pi} \int_{\Sigma} \frac{\partial \varphi}{\partial \nu} dA$, where $\Sigma$ is any surface enclosing $K$ and $\nu$ is the outward-pointing normal derivative.  For example, a ball in Euclidean space has capacity equal to its radius.

Recall that the Poincar\'e--Faber--Szeg\"o inequality (or isocapacitary inequality) bounds the capacity of a set from below in terms of its volume; see \cite[section 3.2]{PS} for a standard proof based on spherical symmetrization.
\begin{thm} [Poincar\'e--Faber--Szeg\"o]
\label{thm_isocap}
Let $K$ be a compact region in Euclidean 3-space. Then
$$\capac(K) \geq \left(\frac{3 |K|}{4\pi}\right)^{1/3}.$$
If equality holds and $\partial K$ is smooth, then $K$ is a closed ball.
\end{thm}
Motivated by this inequality and by Huisken's definition of isoperimetric mass, we propose:
\begin{definition}
\label{def_mCV}
Let $(M,g)$ be an asymptotically flat 3-manifold. Define the \emph{isocapacitary} (or \emph{capacity-volume}) \emph{mass} of $(M,g)$ to be
\begin{equation}
\label{m_CV}
m_{CV}(M,g) = \sup_{\{K_j\}} \limsup_{j \to \infty} \frac{1}{4\pi\capac(K_j)^2} \left[ |K_j|- \frac{4\pi}{3}\capac(K_j)^3  \right],
\end{equation}
where the supremum is taken over all exhaustions $\{K_j\}$ of $M$ by compact sets.
\end{definition}
Later we will show that $m_{CV}$ can be written more cleanly using the difference between the volume radius and the capacity:
\begin{equation}
\label{m_CV2}
m_{CV}(M,g) = \sup_{\{K_j\}} \limsup_{j \to \infty} \left[ \left(\frac{3|K_j|}{4\pi}\right)^{1/3} - \capac(K_j)  \right],
\end{equation}
the form we will prefer throughout the paper.  (Note that expression \eqref{m_CV} more closely resembles Huisken's definition \eqref{miso}, but we point out in the appendix that \eqref{miso} can also be rewritten using twice the difference between the volume radius and area radius. However, the forms \eqref{miso} and \eqref{m_CV} generalize better to higher dimensions; again, see the appendix.)

For example, $m_{CV}(\R^3, \delta_{ij}) \leq 0$ by Theorem \ref {thm_isocap}. That $m_{CV}(\R^3, \delta_{ij}) \geq 0$ can be seen by choosing a sequence of balls of increasing radius. In general, $m_{CV}$ is manifestly well-defined as an extended real number, though it is not obvious that $m_{CV} \neq \pm \infty$.

We point out that although the quantity $\left(\frac{3|K|}{4\pi}\right)^{\frac{1}{3}} - \capac(K) $ depends on both the geometry of $K$ and its complement, i.e. on all of $M$, $m_{CV}$ only depends on the geometry of the end (Lemma \ref{lemma_end}).

Our primary goal is to argue $m_{CV}$ is a reasonable general relativistic mass via the following results. First, $m_{CV}$ is bounded below by the ADM mass:
\begin{thm}
\label{thm_lb} 
Let $(M,g)$ be an asymptotically flat 3-manifold whose scalar curvature is nonnegative outside of a compact set. Then
\begin{equation}
\label{ineq_lb}
m_{CV}(M,g) \geq m_{ADM}(M,g).
\end{equation}
\end{thm}
Second, we show the reverse inequality holds on exhaustions by coordinate balls.
\begin{thm}
\label{thm_ub}
Let $(M,g)$ be an asymptotically flat 3-manifold with nonnegative scalar curvature, with $\partial M$ empty or else consisting of minimal surfaces. Let $B_r$ be the closed coordinate ball of radius $r$. Then
$$\limsup_{r \to \infty}  \left[ \left(\frac{3|B_r|}{4\pi}\right)^{1/3} - \capac(B_r)\right] \leq m_{ADM}(M,g).$$
\end{thm}
Third, we are able to show the reverse of \eqref{ineq_lb} for general exhaustions by assuming stronger asymptotics on $(M,g)$:
\begin{thm}
\label{thm_HF}
Let $(M,g)$ be a Riemannian 3-manifold that is harmonically flat at infinity with nonnegative ADM mass. Then
\begin{equation}
\label{eqn_ub}
m_{CV}(M,g) \leq m_{ADM}(M,g).
\end{equation}
\end{thm}
The precise definition of harmonically flat is recalled in section \ref{sec_HF}. An immediate consequence of Theorems \ref{thm_lb}, \ref{thm_HF}, and the positive mass theorem is:
\begin{cor}
\label{cor_main}
Let $(M,g)$ be a Riemannian 3-manifold that is harmonically flat at infinity, with nonnegative scalar curvature and boundary that is either empty or minimal. Then
$$m_{CV}(M,g) = m_{ADM}(M,g).$$
\end{cor}
In particular, for the Schwarzschild manifold of mass $m > 0$, $m_{CV} = m$.

As with Huisken's isoperimetric mass, the most difficult case seems to be showing the inequality \eqref{eqn_ub} for general exhaustions and general asymptotically flat metrics. We conjecture this to hold under the hypotheses in Theorem \ref{thm_ub}, i.e., the capacity-volume deficit at infinity detects the total mass.

\medskip

Aside from $m_{CV}$ providing another way to understand total mass in general relativity, we suggest that $m_{CV}$ may be a useful tool particularly in the setting of low regularity metrics or low regularity convergence. Just like $m_{iso}$, $m_{CV}$ is defined without use of derivatives of $g$. In particular, it makes sense for $C^0$ Riemannian metrics. We discuss in section \ref{sec_low_reg} some features of the definition that seem to behave more nicely than $m_{iso}$ in lower regularity, considering, for example, $\F$ convergence and $L^p_{\text{loc}}$ convergence.

We point out that the capacity has appeared many times in the study of mass in general relativity, recalling some instances as follows. Bartnik's famous definition of quasi-local mass was presented as a ``nonlinear geometric capacity'' \cite{Ba1}. As part of his proof of the Riemannian Penrose inequality \cite{Bray_RPI}, Bray showed that the ADM mass of an AF 3-manifold of nonnegative scalar curvature and minimal boundary is bounded below by the capacity of the boundary. Bray and Miao generalized this result to produce a corresponding inequality that is valid for non-minimal boundary \cite{BM}. Capacity arises in the study of zero area singularities, which are generalizations of the negative mass Schwarzschild singularity \cites{bray_npms, zas_robbins, zas, J_penrose}. The author observed the difference between the ADM mass and twice the capacity of the boundary is an invariant of the harmonic conformal class of an AF manifold with boundary \cite{J_HCI}. In the conformally flat case, Schwartz as well as Freire and Schwartz proved inequalities relating the ADM mass and the capacity \cites{Schw, FS}. Very recently, Mantoulidis, Miao, and Tam proved some very interesting upper bounds on the capacity in terms of quasi-local mass and established the isocapacitary inequality for Schwarzschild manifolds of positive mass \cite{MMT}. 

\begin{outline}
Section \ref{sec_background} includes some basic lemmas regarding capacity and the capacity-volume mass, as well as a direct demonstration that $m_{CV}$ agrees with $m_{ADM}$ in Schwarzschild metrics of positive mass. The proof of Theorem \ref{thm_lb} (the lower bound on $m_{CV}$) appears in section \ref{sec_lb}. The main ingredients in the proof are the capacity upper bound of Bray and Miao \cite{BM} and the asymptotic estimates of Fan, Shi, and Tam \cite{FST}. 
The proof of Theorem \ref{thm_ub} (the upper bound on $m_{CV}$ when restricting to balls) is in section \ref{sec_ub}. This is the most subtle part of the paper. The approach requires revisiting the proof of the Poincar\'e--Faber--Szeg\"o inequality, a
quasi-local isoperimetric mass estimate of the author and Lee \cite{JL1}, and some estimates for harmonic functions in the AF end. We also give some partial results on the conjectured inequality $m_{CV} \leq m_{ADM}$ without restricting the exhaustions, and in Corollary \ref{cor_expansion}, obtain an expansion of the capacity of a large ball in terms of the radius.  Section \ref{sec_HF} includes the proof of Theorem \ref{thm_HF}, i.e., this inequality for harmonically flat at infinity metrics. In section \ref{sec_low_reg} we discuss $m_{CV}$ for lower regularity metrics and its behavior for $L^p_{\text{loc}}$ convergence, as well as some possible advantages in its definition. We conclude in section \ref{sec_conj} with some conjectures regarding Bernoulli surfaces and isocapacitary regions, based on analogy with constant mean curvature surfaces and isoperimetric regions. A brief appendix shows some facts about Huisken's isoperimetric mass.
\end{outline}

\section{Definition, lemmas, and the Schwarzschild case}
\label{sec_background}
\begin{definition}
\label{def_AF}
A smooth, connected Riemannian $3$-manifold $(M,g)$, possibly with compact boundary, is \emph{asymptotically flat (AF)} if there exists a compact set $K \subset M$ and a diffeomorphism
$\Phi: M \setminus K \to \R^3 \setminus B$, for a closed ball $B$, such that in the \emph{asymptotically flat coordinates} $x=(x^1, x^2, x^3)$ given by $\Phi$, we have
\begin{equation}
\label{decay}
g_{ij} = \delta_{ij} + \sigma_{ij},
\end{equation}
where, for some constant $\tau > \frac{1}{2}$ (the \emph{order}), we have
$$\sigma_{ij} = O(|x|^{-\tau}), \qquad \partial_k\sigma_{ij} =  O(|x|^{-\tau-1}), \qquad \partial_k\partial_\ell \sigma_{ij} = O(|x|^{-\tau-2})$$
as $|x| \to \infty$. Moreover, we require the scalar curvature of $g$ to be integrable.
\end{definition}

\begin{remark}
In \cite{FST}, decay is also assumed for third derivatives of $\sigma_{ij}$. However, the parts of that paper we will use only require decay through second derivatives. In some other cases, cited results assume $\tau=1$, but only $\tau > \frac{1}{2}$ is needed.
\end{remark}

\begin{lemma}
\label{lemma_mcv}
Expressions \eqref{m_CV} and \eqref{m_CV2} for the isocapacitary mass $m_{CV}$ are equal.
\end{lemma}
\begin{proof}
We claim that in the definition \eqref{m_CV} of $m_{CV}$, one may restrict to exhaustions $\{K_j\}$ such that
\begin{equation}
\label{eqn_lim}
\lim_{j \to \infty} \frac{|K_j|}{\capac(K_j)^3} = \frac{4\pi}{3}.
\end{equation}
First, by the isocapacitary inequality (Theorem \ref{thm_isocap}), it can readily be shown for any exhaustion
$$\limsup_{j \to \infty} \frac{|K_j|}{\capac(K_j)^3} \leq \frac{4\pi}{3},$$
since $g$ becomes more and more uniformly equivalent to $\delta$ on $M \setminus K_j$ as $j \to \infty$ by asymptotic flatness. 

Second, for the reverse, we use a similar idea to \cite[Lemma 16]{JL1}.  Let $\{K_j\}$ be an exhaustion of $M$ by compact sets, so $|K_j| \to \infty$. Then for all $j$ sufficiently large, there exists a unique $r_j > 0$ so that the closed coordinate ball $B_{r_j}$ has the same volume as $K_j$. Define
$$K_j' = \begin{cases}
K_j, & \text{ if } \capac(K_j) \leq \capac (B_{r_j})\\
B_{r_j} & \text{ if } \capac(K_j) > \capac (B_{r_j}),
\end{cases}$$
and note $\{K_j'\}$ is an exhaustion of $M$. Then
\begin{align*}
\liminf_{j \to \infty} \frac{|K_j'|}{\capac(K_j')^3} &\geq \liminf_{j \to \infty} \frac{|B_{r_j}|}{\capac(B_{r_j})^3}.
\end{align*}
Since $r_j \to \infty$, the right-hand side equals $\frac{4\pi}{3}$ by asymptotic flatness. Moreover, $K_j'$ has the same volume as $K_j$ and less capacity, so that the sequence $\{K_j'\}$ is a better competitor for the supremum in \eqref{m_CV} than $\{K_j\}$. This shows the claim.

Now, let $\{K_j\}$ be an exhaustion of $M$ by compact sets such that \eqref{eqn_lim} holds. Let $v_j$ be the volume radius and $c_j$ the capacity of $K_j$. In particular, $\displaystyle\lim_{j \to \infty} v_j / c_j = 1$.

Then the term appearing in the definition of $m_{CV}$, i.e. \eqref{m_CV}, is
\begin{align*}
\frac{1}{4\pi\capac(K_j)^2} \left(|K_j|- \frac{4\pi}{3}\capac(K_j)^3  \right) &=\frac{1}{4\pi c_j^2} \left(\frac{4\pi}{3} v_j^3 - \frac{4\pi}{3}c_j^3\right)\\
&= \frac{1}{3c_j^2} (v_j - c_j) (v_j^2 + c_j v_j + c_j^2).
\end{align*}
Since $v_j / c_j$ limits to 1, we have
$$\limsup_{j \to \infty}\frac{1}{4\pi\capac(K_j)^2} \left(|K_j|- \frac{4\pi}{3}\capac(K_j)^3  \right)  = \limsup_{j \to \infty} (v_j - c_j).$$
In particular, ``$\leq$'' holds between \eqref{m_CV} and \eqref{m_CV2}, and the reverse inequality follows from a similar argument.

\end{proof}

Note that the capacity is well-defined for any compact set in a (non-compact) $C^0$ Riemannian manifold $(M,g)$ (i.e., where $M$ is a smooth $n$-manifold and $g$ is a symmetric, positive-definite, continuous covariant 2-tensor), since $\int |\nabla \phi|^2_g dV_g$ continues to be defined. 

\begin{lemma}
\label{lemma_end}
Suppose $(M,g)$ is an asymptotically flat 3-manifold. The value of $m_{CV}(M,g)$ only depends on the AF end. More generally and precisely, we have the following. Let $g'$ be any $C^0$ Riemannian metric on $M$ that is uniformly equivalent to $g$ (possibly $g=g'$).
If $A \supset \partial M$ is any bounded open subset of $M$ with smooth boundary, then $m_{CV}(M,g') = m_{CV}(M \setminus A, g'|_{M \setminus A})$.
\end{lemma}
\begin{proof}
Let $\{K_j\}$ be an exhaustion of $M$ by compact sets. By truncating finitely many terms, we may assume every $K_j$ contains $A$. 
Note that $\tilde K_j = K_j \setminus A$ forms an exhaustion of $M \setminus A$ by compact sets. Since $K_j$ contains $A$, the capacities of $\tilde K_j$ in $M \setminus A$ and of $K_j$ in $M$ are equal with respect to $g'$. On the other hand, their 
volumes differ by a fixed constant, namely $|A|_{g'}$. In particular,
$$|\tilde K_j|_{g'}^{1/3}=|K_j \setminus A|_{g'}^{1/3} = |K_j|_{g'}^{1/3} + O(|K_j|_{g'}^{-2/3}).$$
Since $|K_j|_{g'} \to \infty$ as $j \to \infty$ since $g'$ is uniformly equivalent to an AF metric, we have
$$\limsup_{j \to \infty} \left( \left(\frac{3|K_j|_{g'}}{4\pi}\right)^{1/3} - \capac_{g'}(K_j)  \right) = \limsup_{j \to \infty} \left( \left(\frac{3|\tilde K_j|_{g'}}{4\pi}\right)^{1/3} - \capac_{g'}(\tilde K_j)  \right),$$
which implies the claim $m_{CV}(M,g') \leq m_{CV}(M \setminus A, g'|_{M \setminus A})$. The reverse inequality follows similarly, beginning with an exhaustion of $M \setminus A$ and taking the union of each compact set with $A$.
\end{proof}

The following general smoothing result will be useful.
\begin{lemma}
\label{lemma_smooth}
Let $(M,g)$ be a non-compact $C^0$ Riemannian $n$-manifold.
\begin{enumerate}[(a)]
\item Given a compact set $K \subset M$ and $\epsilon >0$, there exists a compact set $K_\epsilon \supset$ K with smooth boundary such that $|K_{\epsilon} \setminus K|_g < \epsilon$ and $0 \leq \capac_g(K_\epsilon) - \capac_g(K) < \epsilon$.
\item In particular, in Definition \ref{def_mCV} of $m_{CV}$ (for $n=3$), one may without loss of generality consider compact exhaustions $\{K_j\}$ such that each $K_j$ has smooth boundary. 
\end{enumerate}
\end{lemma}

\begin{proof}
First, let $K$ be a compact set in a smooth, non-compact, Riemannian $n$-manifold $(M,g)$. By a well-known theorem of Whitney, there exists a smooth function $f \geq 0$ on $M$ with $f^{-1}(0) = K$. Since $K$ is compact, we may without loss of generality assume $f^{-1}[0,1]$ is compact.

For $\delta \in (0,1)$, $K_\delta=f^{-1}[0,\delta] \supset K$ is compact. It is well known (see \cite[Theorem 4.15]{EG}, for example) that
$$\lim_{\delta \to 0} \capac(K_\delta) = \capac\Bigg(\bigcap_{\delta \in (0,1)} K_\delta\Bigg) = \capac(K).$$
Thus, we can make $0 \leq \capac(K_\delta)- \capac(K)$ as small as desired.
By the co-area formula,
$$|K_\delta| - |K| = \int_0^\delta \int_{f^{-1}(t)} \frac{1}{|\nabla f|} dA dt,$$
so in particular the outer integral is finite for each $\delta$. Then by the dominated convergence theorem, we have $|K_\delta| \to |K|$ as $\delta \to 0$, i.e. we can make $|K_\delta \setminus K|$ as small as desired.

If $g$ is only $C^0$, then it can be approximated by smooth metrics that are arbitrarily uniformly close, so the same properties hold, since volume and capacity are continuous under uniform perturbations of the metric.

Finally, by Sard's theorem $\partial K_\delta = f^{-1}(\delta)$ is smooth for almost all $\delta>0$; with the above, this implies (a).

By applying (a) to each member of an exhaustion of $M$, and letting $\epsilon \to 0$ as $j \to \infty$, (b) follows.
\end{proof}

We conclude this section by demonstrating the plausibility of $m_{CV}$ equaling $m_{ADM}$ in general by explaining why it is true in Schwarzschild manifolds of positive mass. This is a direct argument; the claim also will follow from Corollary \ref{cor_main}.

\begin{prop}
\label{prop_schwarz}
For $m>0$, let
$$g_{ij}=\left(1 + \frac{m}{2r}\right)^4 \delta_{ij}$$
on $\R^3$ minus the open ball of radius $\frac{m}{2}$ about the origin (which we denote by $M$), i.e. $(M,g)$ is the Schwarzschild metric of mass $m$. Then $m_{CV}(M,g) = m$.
\end{prop}
\begin{proof}
Let $B_{r}$ be the set of points with Euclidean radius between $\frac{m}{2}$ and $r> \frac{m}{2}$, endpoints included. Short calculations show
\begin{align*}
\left(\frac{3|B_r|}{4\pi}\right)^{1/3} &= r + \frac{3m}{2} + O(r^{-1})\\
\capac(B_{r}) &= r + \frac{m}{2}.
\end{align*}
(The latter can be shown by observing $\varphi(x) = \frac{1-\frac{r}{|x|}}{1+\frac{m}{2|x|}}$ is the capacitary potential of $B_r$.)
Then
$$\left(\frac{3|B_r|}{4\pi}\right)^{1/3} - \capac(B_r) = m + O(r^{-1}),$$
This implies $m_{CV} \geq m_{ADM}$ for Schwarzschild space. For the reverse inequality, Mantoulidis, Miao, and Tam recently showed that for regions in the Schwarzschild manifold of positive mass that enclose the horizon with a given volume, those bounded by standard coordinate spheres have the least capacity \cite[Theorem 1.10]{MMT}. (This theorem relies on Bray's result that such regions are isoperimetric \cite{bray_thesis}.) This implies $m_{CV} \leq m_{ADM}$.
\end{proof}

\section{Lower bound on capacity-volume mass}
\label{sec_lb}
In this section we prove Proposition \ref{prop_lower} below, which immediately implies Theorem \ref{thm_lb}. To prove this lower bound, we use an exhaustion of $M$ by coordinate balls.

Let $(M,g)$ be an AF 3-manifold of order $\tau > \frac{1}{2}$. Fix an AF coordinate chart $(x^1, x^2, x^3)$ covering $M \setminus K$, and let $\sigma_{ij}$ be as in Definition \ref{def_AF}.  We fix some notation that will be used in the statement and proof.
\begin{itemize}
\item $S_r$ is the coordinate sphere $|x|=r$.
\item $B_r$ is the compact region enclosed by $S_r$.
\item $A(r)$ is the $g$-area of $S_r$.
\item $V(r)$ is the $g$-volume of $B_r$.
\item $H_r$ is the mean curvature of $S_r$ with respect to $g$ (with the sign convention so that $H_r>0$ for large $r$).
\end{itemize}

\begin{prop}
\label{prop_lower}
Let $(M,g)$ be an asymptotically flat 3-manifold whose scalar curvature is nonnegative outside of a compact set. Then
$$\lim_{r \to \infty}\left[ \left(\frac{3V(r)}{4\pi}\right)^{1/3} - \capac(B_r)\right] \geq m_{ADM}(M,g).$$
In particular, $m_{CV}(M,g) \geq m_{ADM}(M,g).$
\end{prop}

\begin{proof}
We must quantify the asymptotic behavior of $V(r)$ and of $\capac(B_r)$. 

Let $m$ denote the ADM mass of $(M,g)$. By \cite[eq. (2.28)]{FST}, 
$$V(r) = \frac{1}{2}r A(r) - \frac{2\pi r^3}{3} + 2\pi m r^2 + o(r^2)$$
as $r \to \infty$. 
Also, by \cite[Lemma 2.1]{FST}
\begin{equation}
\label{eqn_Ar}
A(r) = 4\pi r^2 + 4\pi \beta(r) + O(r^{2-2\tau}),
\end{equation}
where $\beta(r)$ is $O(r^{2-\tau})$, given explicitly by
\begin{equation}
\label{eqn_beta}
\beta(r) = \frac{1}{8\pi} \int_{S_r} h^{ij} \sigma_{ij} dA_r,
\end{equation}
where $dA_r$ is the area form on $S_r$ with respect to $g$ and $h^{ij}$ is $h_{ij}$ raised with respect to $g^{ij}$, with $h_{ij}$ being the induced metric on $S_r$ expressed in the coordinate chart.
Combining these,
\begin{align*}
V(r) &= \frac{1}{2}r \left(4\pi r^2 + 4\pi \beta(r) + O(r^{2-2\tau})\right)  - \frac{2\pi r^3}{3} + 2\pi m r^2 + o(r^2)\\
&= \frac{4}{3}\pi r^3 + 2\pi m r^2 + 2\pi r\beta(r) + o(r^2).
\end{align*}
Then
\begin{align}
\left(\frac{3V(r)}{4\pi}\right)^{1/3} &= r\left(1 + \frac{m}{2r} + \frac{\beta(r)}{2r^2} + o(r^{-1})\right) \nonumber \\
&= r + \frac{m}{2} +  \frac{\beta(r)}{2r} + o(1). \label{eqn_vol_estimate}
\end{align}
Note the $\beta(r)r^{-1}$ term is potentially problematic, as a priori it can be larger than the relevant constant term $\frac{m}{2}$. 

\begin{remark}
\label{rmk_beta} In a Schwarzschild (or more generally, harmonically flat) manifold, $\frac{\beta(r)}{2r} $ is equal to $m + O(r^{-1})$.
\end{remark}

Next we turn our attention to $\capac(B_r)$.  We will use the upper bound on capacity due to Bray and Miao \cite{BM} in terms of the area and Willmore energy of the boundary (generalizing the case of a minimal boundary due to Bray \cite{Bray_RPI}).
\begin{thm}[Bray--Miao, Theorem 1 of \cite{BM} ]
\label{thm_bray_miao}
Let $(M,g)$ be an asymptotically flat 3-manifold with nonnegative scalar curvature with connected, nonempty boundary $\partial M$.
  Assume $M$ is diffeomorphic to the complement of a bounded domain in $\R^3$. Then
$$\capac(\partial M) \leq \sqrt{\frac{|\partial M|}{16\pi}} \left( 1 + \sqrt{\frac{1}{16\pi} \int_{\partial M} H^2 dA}\;\right),$$
where $H$ and $dA$ are the mean curvature and area form of $\partial M$. 
If equality holds, then $(M,g)$ is isometric to a rotationally symmetric subset of a Schwarzschild manifold. 
\end{thm}

Note that by the hypotheses of Proposition \ref{prop_lower}, for all $r$ sufficiently large we may apply  Theorem \ref{thm_bray_miao} to $M \setminus B_r$, with the restriction of $g$ thereto. Clearly, the capacity of the boundary of this manifold equals $\capac(B_r)$ in $M$. Then:
\begin{equation}
\label{bray_miao_ineq}
\capac(B_r) \leq \sqrt{\frac{A(r)}{16\pi}} \left(1 + \sqrt{\frac{1}{16\pi} \int_{S_r} H_r^2 dA_r}\;\right).
\end{equation}
We therefore are interested in the asymptotics of the expression on the right. With an estimate of $A(r)$ already in hand, we study the Willmore energy term.

From \cite[Lemma 2.1]{FST} we have $H_r = \frac{2}{r} + \alpha(r),$ where $\alpha(r)$ is $O(r^{-1-\tau})$. Then
$$\int_{S_r} H_r^2 dA_r = \frac{4A(r)}{r^2} + \frac{4}{r} \int_{S_r} \alpha(r) dA_r + \int_{S_r} \alpha(r)^2 dA_r.$$
On the other hand,
\begin{align*}
\frac{1}{A(r)} \left( \int_{S_r} H_r dA_r\right)^2 &= \frac{1}{A(r)} \left(\frac{2}{r} A(r) +  \int_{S_r} \alpha(r) dA_r\right)^2\\
&=\frac{4A(r)}{r^2}  + \frac{4}{r} \int_{S_r} \alpha(r) dA_r + \frac{1}{A(r)}\left(\int_{S_r} \alpha(r) dA_r\right)^2.
\end{align*}
Thus,
$$\int_{S_r} H_r^2 dA_r - \frac{1}{A(r)} \left( \int_{S_r} H_r dA_r\right)^2 = O(r^{-2\tau}).$$
By \cite[Lemma 2.2]{FST},
$$\int_{S_r} H_r dA_r = \frac{A(r)}{r} + 4\pi r -8\pi m +o(1),$$
so
$$\left(\int_{S_r} H_r dA_r\right)^2 = \frac{A(r)^2}{r^2} + 8 \pi A(r) + 16 \pi^2 r^2 - \frac{16\pi m A(r)}{r} -64\pi^2 mr + o(r),$$
and
$$\frac{1}{A(r)} \left(\int_{S_r} H_r dA_r\right)^2 = \frac{A(r)}{r^2} + 8 \pi  + \frac{16 \pi^2 r^2}{A(r)} - \frac{16\pi m}{r} -\frac{64\pi^2 mr}{A(r)} + o(r^{-1}).$$
Then
$$\frac{1}{16\pi} \int_{S_r} H_r^2 dA_r = \frac{A(r)}{16\pi r^2} + \frac{1}{2}  + \frac{\pi r^2}{A(r)} - \frac{m}{r} -\frac{4\pi mr}{A(r)} + o(r^{-1}),$$
since $2\tau > 1$. We combine this with the expansion \eqref{eqn_Ar} for $A(r)$ to obtain
$$\sqrt{\frac{1}{16\pi} \int_{S_r} H_r^2 dA_r} = 1 - \frac{m}{r} + o(r^{-1}).$$
Finally,
\begin{align}
\sqrt{\frac{A(r)}{16\pi}} \left(1 + \sqrt{\frac{1}{16\pi} \int_{S_r} H_r^2 dA_r}\right) &= \sqrt{\frac{r^2}{4} + \frac{\beta(r)}{4} + O(r^{2-2\tau})}\left( 2 - \frac{m}{r} + o(r^{-1})\right) \nonumber\\
&= \frac{r}{2} \sqrt{1 + \frac{\beta(r)}{r^2} + O(r^{-2\tau})}\left( 2 - \frac{m}{r} + o(r^{-1})\right) \nonumber\\
&= \frac{r}{2} \left(1 + \frac{\beta(r)}{2r^2} + O(r^{-2\tau})\right)\left( 2 - \frac{m}{r} + o(r^{-1})\right) \nonumber\\
&= r +\frac{\beta(r)}{2r}- \frac{m}{2} + o(1). \label{eqn_ah}
\end{align}
Combining \eqref{eqn_vol_estimate} with \eqref{bray_miao_ineq} and \eqref{eqn_ah}, Proposition \ref{prop_lower} follows, noting the convenient cancellation of the $\beta(r)r^{-1}$ terms.
\end{proof}

It would be interesting to determine whether nonnegative scalar curvature outside of a compact set is a necessary hypothesis for the inequality $m_{CV} \geq m_{ADM}$. (We suspect it is not.) For comparison, the inequality $m_{iso} \geq m_{ADM}$ does not require a sign on the scalar curvature \cite{FST}. (Note: the definition of isoperimetric mass in \cite{FST} is different, as it assumes the exhaustions are balls. The proof therein, due to Miao, implies $m_{iso} \geq m_{ADM}$ for the definition we use here.)

\section{Upper bound on $m_{CV}$ for exhaustions by balls}
\label{sec_ub}
In this section we prove Theorem \ref{thm_ub}, i.e. when restricting to exhaustions by balls, the ADM mass is an upper bound for the capacity-volume mass.

One key ingredient in the proof will be the following estimate for the (quasi-local) isoperimetric mass, due to the author and Lee, presented here in a slightly simplified form. (We also point out the related result \cite[Corollary C.4]{CESY}.) Recall the isoperimetric ratio of a bounded open set $\Omega$ in a Riemannian 3-manifold  is $\frac{|\partial^* \Omega|^{3/2}}{6 \sqrt{\pi}|\Omega|}$, and the isoperimetric constant of the manifold is the infimum of the isoperimetric ratios of all such sets.

\begin{thm}[Theorem 17 of \cite{JL1}]
\label{thm_JL}
Given positive constants $\mu_0, I_0,$ and $c_0$, there exists a constant $C$ depending only on $\mu_0, I_0,$ and $c_0$ with the following property. Let $(M,g)$ be a smooth asymptotically flat 3-manifold  with  nonnegative scalar curvature and no compact minimal surfaces in its interior, whose boundary is empty or minimal, with $m_{ADM}(M,g) \leq \mu_0$. Let $\Omega$ be an outward-minimizing bounded open set in $M$ with $C^{1,1}$ boundary $\partial \Omega$, such that $\Omega \supset \partial M$. Assume that $|\partial \Omega| \geq 36\pi \mu_0^2$, that the isoperimetric ratio of $\Omega$ is at most $I_0$, and that the isoperimetric constant of $(M,g)$ is at least $c_0$. 
Then 
\begin{equation}
\label{eqn_JL}
\frac{2}{|\partial \Omega|} \left[|\Omega| - \frac{1}{6\sqrt{\pi}} |\partial \Omega|^{3/2}\right] \leq m_{ADM}(M,g) + \frac{C}{\sqrt{|\partial \Omega|}}.
\end{equation}
\end{thm}
\begin{remark}
\label{remark_JL}
Recall that a bounded open set $\Omega$ in $(M,g)$ admits a minimizing hull $\tilde \Omega \supseteq \Omega$, i.e. a maximal open set containing $\Omega$ of the least possible perimeter. If $\partial \Omega$ is smooth, then $\partial \tilde \Omega$ is $C^{1,1}$ surface (we refer the reader to section 1 of \cite{HI} for details). In this case, if $\Omega$ fails to be outward-minimizing, then $\tilde \Omega$ satisfies the hypotheses of the theorem, provided $|\partial \tilde \Omega| \geq 36\pi \mu_0^2$. In particular, since $\tilde \Omega$ has greater volume and less perimeter than $\Omega$, \eqref{eqn_JL} holds for $\Omega$ as well (with the same constant $C$). Note also we can use bounded open and compact sets interchangeably, by the regularity assumptions on the boundary.
\end{remark}

In Lemma \ref{lemma_JL} in the appendix, we will show the constant $C$ can be chosen to be $\gamma_0 \mu_0^2(1 + I_0^2 c_0^{-3})$, where $\gamma_0$ is a universal constant. When we use the theorem, we will simply take $\mu_0$ and $c_0$ to be the ADM mass and isoperimetric constant of $(M,g)$ itself; i.e. only the isoperimetric ratio will require additional control.

\begin{proof}[Proof of Theorem \ref{thm_ub}]
First, we point out a slight discrepancy in the hypotheses of Theorems \ref{thm_ub} and \ref{thm_JL}: the latter requires no interior minimal surfaces. However, as $m_{CV}$ (and obviously $m_{ADM})$ only depend on the geometry of the end (Lemma \ref{lemma_end}), we can without loss of generality replace $(M,g)$ in Theorem \ref{thm_ub} with the region exterior to (and including) the outermost minimal surface (i.e., delete the ``trapped region'' --- see \cite[Lemma 4.1]{HI}).

By the positive mass theorem, we may assume without loss of generality that $m_{ADM}(M,g)>0$, as the zero mass (hence, Euclidean case) is trivial.

For this direction, we seek a lower bound on capacity. Thus we are inspired by the isocapacitary inequality for Euclidean space and attempt to generalize the proof, beginning by following \cite[section 3.2]{PS}. The proof we present will soon deviate once we encounter the isoperimetric inequality; we will need to use Theorem \ref{thm_JL} on the level sets of the potential, which will entail obtaining precise control on these level sets.

Let $\Omega \subset M$ be a compact set with smooth boundary, whose interior contains $\partial M$. Let $\varphi$ be the capacitary potential of $\Omega$, i.e. the unique $g$-harmonic function on $M \setminus \mathring\Omega$ vanishing on $\partial \Omega$ and tending to 1 at infinity. In this proof, geometric quantities will be with respect to $g$, unless they are denoted with a $0$ subscript, indicating they are with respect to the Euclidean metric in the coordinate chart.

By the co-area formula
$$\capac(\Omega) = \frac{1}{4\pi} \int_{M \setminus \Omega} |\nabla \varphi|^2 dV =\frac{1}{4\pi}  \int_0^1 \int_{\Sigma_t} |\nabla \varphi| dA dt,$$
where $\Sigma_t = \varphi^{-1}(t)$, which is a smooth, compact surface for almost all $t$.  The Cauchy--Schwarz inequality gives
$$|\Sigma_t|^2 \leq \left( \int_{\Sigma_t} |\nabla \varphi| dA\right) \left( \int_{\Sigma_t} \frac{1}{|\nabla \varphi|} dA\right),$$
so
$$\capac(\Omega) \geq  \frac{1}{4\pi} \int_0^1\frac{|\Sigma_t|^2}{\int_{\Sigma_t} \frac{1}{|\nabla \varphi|} dA}  dt.$$
The next step would be to apply the isoperimetric inequality. We postpone this by inserting the isoperimetric ratio 
$$I(t) = \frac{|\Sigma_t|^{3/2}}{6\sqrt{\pi}V(t)}$$
where $V(t) = |\Omega_t|$ and $\Omega_t$ is the compact region bounded by $\Sigma_t$, i.e. $\Omega_t = \Omega \cup \varphi^{-1}[0,t]$. By the co-area formula,
$$V(t) = |\Omega| + \int_{0}^t \int_{\Sigma_s} \frac{1}{|\nabla \varphi|} dA ds,$$
so for almost all $t$,
$$V'(t) = \int_{\Sigma_t} \frac{1}{|\nabla \varphi|} dA.$$
Then
\begin{align}
\capac(\Omega) &\geq  \frac{\left(6\sqrt{\pi}\right)^{4/3}}{4\pi} \int_0^1\frac{V(t)^{4/3} I(t)^{4/3} }{V'(t)}  dt \nonumber\\
&=\underbrace{\frac{\left(6\sqrt{\pi}\right)^{4/3}}{4\pi} \int_0^1\frac{V(t)^{4/3} }{V'(t)}  dt}_{\rm \mathbf I} - \underbrace{\frac{\left(6\sqrt{\pi}\right)^{4/3}}{4\pi} \int_0^1\frac{V(t)^{4/3} \left(1-I(t)^{4/3}\right) }{V'(t)}  dt}_{\rm \mathbf {II}}. \label{two_integrals}
\end{align}
Expression ${\rm \mathbf {I}}$ can be estimated from below as follows (completely following the Euclidean case):
define $R(t)$ via $\frac{4}{3} \pi R(t)^3 = V(t)$, so $V'(t) = 4\pi R(t)^2 R'(t)$ for almost all $t$. Then
\begin{align*}
 \int_0^1\frac{V(t)^{4/3} }{V'(t)}  = \frac{\left(\frac{4\pi}{3}\right)^{4/3}}{(4\pi)^2}\int_0^1 \frac{4\pi R(t)^2}{R'(t)} dt. 
\end{align*}
On $\R^3$ let $\tilde S_r$ denote the sphere of radius $r$ about the origin, and let $\tilde \varphi(t)$ be the (Lipschitz) function that equals $t$ on $\tilde S_{R(t)}$. In particular, $\tilde \varphi(t)$ is a valid test function for the Euclidean capacity of the ball $\tilde B_{R(0)}$ of radius $R(0)$, and $|\nabla \tilde \varphi(t)|_0 = \frac{1}{R'(t)}$ for almost all $t$. Then using the co-area formula, we have

\begin{align}
\rm \mathbf {I} &= \frac{\left(6\sqrt{\pi}\right)^{4/3}}{4\pi} \cdot \frac{\left(\frac{4\pi}{3}\right)^{4/3}}{(4\pi)^2}\int_0^1 \frac{4\pi R(t)^2}{R'(t)} dt\nonumber\\
&= \frac{1}{4\pi} \int_0^1 \int_{\tilde S_{R(t)}} |\nabla \tilde \varphi(t)|_0 dA_0 dt\nonumber\\
&= \frac{1}{4\pi} \int_{\R^3 \setminus \tilde B_{R(0)}} |\nabla \tilde \varphi|_0^2 dV_0\nonumber\\
&\geq \capac_0 (\tilde B_{R(0)})\nonumber\\
&= R(0)\nonumber\\
&= \left(\frac{3|\Omega|}{4\pi}\right)^{1/3} \label{eqn_I}
\end{align}
since the capacity and the volume radius agree for Euclidean balls.

To address the term ${\rm \mathbf {II}}$ in \eqref{two_integrals}, we first need to understand the isoperimetric ratio $I(t)$. For this, we will apply Theorem \ref{thm_JL} to the $\Omega_t$. As mentioned earlier, we will use $\mu_0=m_{ADM}(M,g)>0$ and $c_0$ the (positive) isoperimetric constant of $(M,g)$. Then in light of Lemma \ref{lemma_JL} in the appendix and Remark \ref{remark_JL} we have (letting $m=m_{ADM}(M,g)$)
\begin{equation}
\label{eqn_JL_est}
\frac{2}{|\Sigma_t|} \left( V(t) - \frac{1}{6\sqrt{\pi}}|\Sigma_t|^{3/2}\right) \leq m + \frac{\gamma_1 I(t)^2}{\sqrt{|\Sigma_t|}}
\end{equation}
for almost all $t$, for a constant $\gamma_1$ depending only on $m$ and $c_0$, provided the least area needed to enclose $\Sigma_t$ is $\geq 36\pi m^2$. (This will be verified later; we assume it to hold for now.) Below, $\gamma_2$ and $\gamma_3$ will be constants depending only on $m$ and $c_0$.
Then
$$\frac{1}{I(t)} -1\leq \frac{3\sqrt{\pi} m}{\sqrt{|\Sigma_t|}} + \frac{3\sqrt{\pi}\gamma_1 I(t)^2}{|\Sigma_t|}.$$
so
\begin{align*}
1-I(t)^{4/3}  &\leq 1- \frac{1}{\left(1+ \frac{3\sqrt{\pi} m}{\sqrt{|\Sigma_t|}} + \frac{3\sqrt{\pi}\gamma_1 I(t)^2}{|\Sigma_t|}\right)^{4/3}}\\
&\leq \frac{4\sqrt{\pi} m}{\sqrt{|\Sigma_t|}} + \frac{\gamma_2 I(t)^2}{|\Sigma_t|}.
\end{align*}

Thus,
\begin{equation}
\label{eqn_II}
{\rm \mathbf {II}} \leq \underbrace{\frac{\left(6\sqrt{\pi}\right)^{4/3}}{4\pi} (4\sqrt{\pi} m) \int_0^1\frac{V(t)^{4/3} }{\sqrt{|\Sigma_t|}V'(t)}  dt}_{\rm \mathbf {III}} +  \gamma_3\underbrace{ \int_0^1\frac{I(t)^2V(t)^{4/3} }{|\Sigma_t|V'(t)}  dt}_{\rm \mathbf {IV}}.
\end{equation}
The goal of the proof is to show ${\rm \mathbf {III}}$ is the ADM mass $m$ modulo decaying terms, and ${\rm \mathbf {IV}}$ decays (as $|\Omega|$ gets very large).

\medskip

Up to this point, the argument applied to a general region $\Omega$, but now we specialize to the case in which $\Omega$ is a large coordinate ball. This will allow us to gain enough control on the level sets of the potential to unravel ${\rm \mathbf {III}}$ and ${\rm \mathbf {IV}}$. We proceed to gather some estimates before doing so.

Assume $\rho_0$ is sufficiently large  so that $S_{\rho}$ is a well-defined, outward-minimizing coordinate sphere $\{|x|=\rho\}$  in $M$ for all $\rho \geq \frac{1}{2}\rho_0$. (The outward-minimizing condition can be assured because $M$ is foliated near infinity by coordinate spheres of positive mean curvature.) Let $B_{\rho}$ be the compact region bounded by $S_{\rho}$. Assume $\rho \geq \rho_0$. (Throughout the proof we will increase $\rho_0$ and always assume $\rho \geq \rho_0$.)
Let $\varphi=\varphi_{\rho}$ be the capacitary potential for $\Omega=B_{\rho}$. From asymptotic flatness, once would expect $B_{\rho}$ to have capacity approximately equal to $\rho$ for large $\rho$, so that $\varphi_{\rho}$ will approximately have the expansion $1-\frac{\rho}{|x|} + \ldots$ at infinity. To make this precise, from Lemma \ref{lemma_potential} below, we may increase $\rho_0$ if necessary and find functions $a_1 = \rho - k_1\rho^{1-\tau}$ and $a_2 = \rho + k_1 \rho^{1-\tau}$ (for a constant $k_1>0$ independent of $\rho$) so that
\begin{align}
|c(\rho) -\rho|& \leq k_1\rho^{1-\tau}\nonumber\\
\label{phi_est}
\frac{a_1}{|x|} \leq 1-\varphi_{\rho}(x) &\leq \frac{a_2}{|x|}\\
\frac{a_1}{|x|^2} \leq |\nabla \varphi_{\rho}(x) |&\leq \frac{a_2}{|x|^2}\label{dphi_est}
\\
\left|\Hess(\varphi_\rho)\left(\frac{\nabla \varphi_\rho}{|\nabla \varphi_\rho|},\frac{\nabla \varphi_\rho}{|\nabla \varphi_\rho|} \right) +\frac{2c(\rho)^3}{|x|^7}\right| &\leq \frac{k_1\rho^3}{|x|^{7+\tau}}  \label{ddphi_est}
\end{align}
for $x \in M \setminus \mathring B_{\rho}$, where $c(\rho)=\capac(B_\rho)$. Later we will repeatedly use the fact any pair of the quantities $\rho, c(\rho), a_1$, and $a_2$ has its ratio limiting to 1 as $\rho \to \infty$.

If necessary, increase $\rho_0$ so that 
\begin{equation}
\label{2r1}
a_1 \geq \frac{1}{2}\rho
\end{equation}
for all $\rho \geq \rho_0$. In particular, by \eqref{dphi_est}, all of the level sets $\Sigma_t$ of $\varphi_{\rho}$ are smooth. (We omit ``$\rho$'' from the notation for $\Sigma_t$ and $\Omega_t$.)

From \eqref{phi_est} it follows that
$$B_{\frac{a_1}{1-t}} \subseteq \Omega_t \subseteq B_{\frac{a_2}{1-t}}.$$
In particular, 
\begin{equation}
\label{vol_ub}
\left|B_{\frac{a_1}{1-t}}\right| \leq V(t) \leq \left|B_{\frac{a_2}{1-t}}\right|
\end{equation}
and $\Sigma_t$ encloses the coordinate sphere $S_{\frac{a_1}{1-t}}$. 
Moreover, since $\frac{a_1}{1-t} \geq a_1  \geq \frac{1}{2}\rho$, $S_{\frac{a_1}{1-t}}$ is outward-minimizing with respect to $g$. Then
\begin{equation}
\label{area_lb}
|\Sigma_t| \geq \left|S_{\frac{a_1}{1-t}}\right|.
\end{equation}
(An upper bound on $|\Sigma_t|$ will be found later; this is more subtle.)
These inequalities will ultimately be used in \eqref{eqn_II}. We also need:
\begin{equation}
\label{eqn_V_prime}
V'(t)= \int_{\Sigma_t} \frac{1}{|\nabla \varphi_\rho|} dA \geq \int_{\Sigma_t} \frac{|x|^2}{a_2} dA \geq \frac{a_1^2}{a_2 (1-t)^2} |\Sigma_t|
\end{equation}
which follows from the the co-area formula, \eqref{dphi_est}, and \eqref{phi_est}.

To use these estimates in ${\rm \mathbf {III}}$ and ${\rm \mathbf {IV}}$, we will relate volumes and areas back to the Euclidean metric. From the AF decay of $g$, it is straightforward to estimate the relatives sizes of $B_\rho$ and $S_\rho$ with respect to $g$ and  their Euclidean counterparts. (This is recalled more precisely than needed here in the proof of Proposition \ref{prop_lower}.) In particular, there exists a constant $k_2>0$ depending only on $g$, so that for $\rho \geq \rho_0$,
\begin{equation}
\label{vol}
\frac{4}{3}\pi \rho^3 (1-k_2 \rho^{-\tau}) \leq |B_\rho| \leq \frac{4}{3}\pi \rho^3 (1+k_2 \rho^{-\tau})
\end{equation}
and
\begin{equation}
\label{area}
 4\pi \rho^2(1-k_2 \rho^{-\tau}) \leq |S_\rho| \leq 4\pi \rho^2(1+k_2 \rho^{-\tau}).
\end{equation}
Now, let $\epsilon \in (0,1)$ be arbitrary. By increasing $\rho_0$ if necessary, we can be sure that $k_2 \rho^{-\tau} < \epsilon$ for all $\rho \geq \frac{1}{2}\rho_0$. Then combining \eqref{vol_ub} with \eqref{vol} (and using \eqref{2r1}) we obtain
\begin{equation}
\label{V_upper}
V(t) \leq \frac{4\pi}{3} \frac{a_2^3}{(1-t)^3} \left( 1 + k_2  \left(\frac{a_2}{1-t}\right)^{-\tau}\right) \leq \frac{4\pi}{3} \frac{a_2^3}{(1-t)^3} \left(1+\epsilon\right),
\end{equation}
and likewise
\begin{equation}
\label{V_lower}
V(t) \geq \frac{4\pi}{3} \frac{a_1^3}{(1-t)^3} \left(1-\epsilon\right).
\end{equation}
We also combine \eqref{area_lb} with \eqref{area} to obtain
\begin{equation}
\label{A_lower}
|\Sigma_t| \geq 4\pi \frac{a_1^2}{(1-t)^2} \left(1-k_2 \left(\frac{a_1}{1-t}\right)^{-\tau}\right) \geq 4\pi \frac{a_1^2}{(1-t)^2} \left(1-\epsilon\right).
\end{equation}

Now we can verify the lower bound on the minimum enclosing area of $\Sigma_t$ (which was needed to assert \eqref{eqn_JL_est}). Since
$S_{\frac{a_1}{1-t}}$ is outward-minimizing with respect to $g$, and $\Sigma_t$ encloses this surface, the least $g$-area needed to enclose $\Sigma_t$ is at least
$$\left|S_{\frac{a_1}{1-t}}\right|_g \geq 4\pi \frac{a_1^2}{(1-t)^2} \left(1-k_2 \left(\frac{a_2}{1-t}\right)^{-\tau}\right) \geq 4\pi a_1^2 (1-\epsilon).$$
Again increasing $\rho_0$ if necessary we can be sure this is at least $36\pi m^2$ for all $\rho \geq \rho_0$ and all $t \in [0,1)$.

Thus, we may apply Theorem \ref{thm_JL} to $\Omega_t$; in particular \eqref{eqn_II} holds. We can also now estimate ${\rm \mathbf {III}}$ as (using \eqref{V_upper}, \eqref{eqn_V_prime}, and \eqref{A_lower}):
\begin{align*}
{\rm \mathbf {III}} = \frac{\left(6\sqrt{\pi}\right)^{4/3}}{4\pi} (4\sqrt{\pi} m) \int_0^1\frac{V(t)^{4/3} }{\sqrt{|\Sigma_t|}V'(t)}  dt &\leq \kappa \int_0^1 \frac{\frac{a_2^4}{(1-t)^4} \left(1+\epsilon\right)^{4/3} }{\frac{a_1^2}{a_2 (1-t)^2}|\Sigma_t|^{3/2}}dt\\
&\leq\frac{\kappa}{(4\pi)^{3/2}}\int_0^1 \frac{\frac{a_2^4}{(1-t)^4} \left(1+\epsilon\right)^{4/3} }{\frac{a_1^2}{a_2 (1-t)^2} \frac{a_1^3}{(1-t)^3}\left(1-\epsilon\right)^{3/2}}dt\\
&= (2m)\frac{ \left(1+\epsilon\right)^{4/3}a_2^5}{\left(1-\epsilon\right)^{3/2}a_1^5}  \int_0^1 (1-t) dt,
\end{align*}
where
$$\kappa = \frac{\left(6\sqrt{\pi}\right)^{4/3}}{4\pi} (4\sqrt{\pi} m) \left(\frac{4\pi}{3}\right)^{4/3}.$$
In particular, $\displaystyle \limsup_{\rho \to \infty}$ of ${\rm \mathbf {III}}$ is at most $\frac{ \left(1+\epsilon\right)^{4/3}}{\left(1-\epsilon\right)^{3/2}} m$, since $\frac{a_2}{a_1}$ limits to 1.

Expression ${\rm \mathbf {IV}}$ is similar to ${\rm \mathbf {III}}$, except for the (favorable) additional $\sqrt{|\Sigma_t|}$ factor in the denominator and the (unfavorable) dependence on the isoperimetric ratio in the numerator. We will handle this by finding an appropriate upper bound on the area of the level sets:

\begin{lemma}
\label{lemma_isop}
For any $q >0$, we may increase $\rho_0$ if necessary to ensure
$$|\Sigma_t| \leq \frac{|S_{\rho}|}{(1-t)^{2+q}}$$
for all level sets $\Sigma_t$ of $\varphi_{\rho}$, and for all $\rho \geq \rho_0$.
\end{lemma}
\begin{proof}
We approach this by bounding the mean curvature of the level sets. 

Recall that every level set $\Sigma_t$ of $\varphi_{\rho}$ is regular. By decomposing the Laplacian of $\varphi_{\rho}$ on $M$ into the tangential and normal parts to $\Sigma_t$, we have
$$\cancel{\Delta_M u} = \cancel{\Delta_{\Sigma_t} u} + \Hess(\varphi_{\rho})(\nu, \nu) + H \nu(\varphi_{\rho}),$$
on each level set, where $H=H_t^{\rho}$ is the mean curvature of $\Sigma_t$ and $\nu = \frac{\nabla \varphi_{\rho}}{|\nabla \varphi_{\rho}|}$ is the unit normal to $\Sigma_t$. In particular, using \eqref{dphi_est} and \eqref{ddphi_est},
\begin{align*}
|H_t^{\rho}(x)| &= \frac{|\Hess(\varphi_{\rho})(\nabla \varphi_{\rho}, \nabla \varphi_{\rho})|}{|\nabla \varphi_{\rho}|^3}\\
&\leq \frac{|x|^6}{a_1^3} \left(\frac{2c(\rho)^3}{|x|^7} + \frac{k_1\rho^3}{|x|^{7+\tau}} \right).
\end{align*}
Since $\frac{c(\rho)}{a_1}$ limits to 1 as $\rho \to \infty$ and $|x| \geq \rho$, we can increase $\rho_0$ if necessary to assure
$$|H_t^{\rho}(x)| \leq \frac{2+\frac{q}{2}}{|x|}.$$
Using the first variation of area followed by \eqref{dphi_est} and \eqref{phi_est}, 
\begin{align*}
\frac{d}{dt} |\Sigma_t| &= \int_{\Sigma_t} \frac{H_t^{\rho}}{|\nabla \varphi_{\rho}|} dA\\
&\leq \int_{\Sigma_t} \frac{(2+\frac{q}{2})|x|}{a_1}dA\\
&\leq \frac{(2+\frac{q}{2})a_2}{(1-t)a_1} |\Sigma_t|
\end{align*}
Again increasing $\rho_0$ and $\rho$ if necessary, we may assume
\begin{align*}
\frac{d}{dt} |\Sigma_t|_g &\leq \frac{2+q}{1-t} |\Sigma_t|.
\end{align*}
By an elementary comparison argument,
$$|\Sigma_t| \leq \frac{|\Sigma_0|}{(1-t)^{2+q}},$$
and of course $\Sigma_0 = S_{\rho}$.
\end{proof}

Then we can estimate:
\begin{align*}
{\rm \mathbf {IV}} &= \int_0^1 \frac{I(t)^2}{\sqrt{|\Sigma_t|}}\cdot \frac{V(t)^{4/3} }{\sqrt{|\Sigma_t|} V'(t)}  dt\\
&= \frac{1}{36\pi}\int_0^1 \frac{|\Sigma_t|^{5/2} }{V(t)^2}\cdot \frac{V(t)^{4/3} }{\sqrt{|\Sigma_t|} V'(t)}  dt\\
&\leq c \int_0^1 \frac{|S_{\rho}|^{5/2} (1-t)^6 }{a_1^6(1-t)^{\frac{5}{2}(2+q)}}\cdot \frac{V(t)^{4/3} }{\sqrt{|\Sigma_t|} V'(t)}  dt\\
&= O(\rho^{-1})\cdot {\rm \mathbf {III}}
\end{align*}
for some constant $c$ independent of $\rho$, having used \eqref{V_lower}; applied Lemma \ref{lemma_isop} with $q=\frac{2}{5}$; and noted $|S_\rho|^{5/2} a_1^{-6}$ is $O(\rho^{-1})$. Thus $\displaystyle\lim_{\rho \to \infty} {\rm \mathbf {IV}}=0$.

Finally, to put it all together, from \eqref{two_integrals}, \eqref{eqn_I}, and \eqref{eqn_II},
\begin{align*}
\left(\frac{3|B_{\rho}|}{4\pi}\right)^{1/3} - \capac(B_{\rho}) &\leq {\rm \mathbf {II}}\\
&\leq {\rm\mathbf {III}} + {\rm \mathbf {IV}}.
\end{align*}
Then by our estimates for ${\rm\mathbf {III}}$ and ${\rm\mathbf {IV}}$,
$$\limsup_{\rho \to \infty} \left[\left(\frac{3|B_{\rho}|}{4\pi}\right)^{1/3} - \capac(B_{\rho})\right] \leq \frac{ \left(1+\epsilon\right)^{4/3}}{\left(1-\epsilon\right)^{3/2}} m.$$
Since $\epsilon $ with arbitrary, the proof of Theorem \ref{thm_ub} is complete (except for the proof of Lemma \ref{lemma_potential}).
\end{proof}

It is well known that a capacitary potential in an AF manifold admits a ``$\frac{1}{|x|}$'' expansion at infinity; the following lemma used in the proof of Theorem \ref{thm_ub}  establishes a type of uniform control on these potentials for balls as the radii run off to infinity, e.g. the error terms are uniformly controlled modulo a factor of $\rho$.
\begin{lemma}
\label{lemma_potential}
Let $(M,g)$ be an asymptotically flat 3-manifold of order $\frac{1}{2} < \tau < 1$. Fix a corresponding AF coordinate chart on $M$. Then there exists $\rho_0>0, k_1>0$ such that for all $\rho \geq \rho_0$, the capacitary potential $\varphi_{\rho}$ of the coordinate ball $B_{\rho}$ satisfies
$$\varphi_{\rho}(x) = 1 - \frac{c(\rho)}{|x|} + W_{\rho}(x),$$
where $c(\rho)$ is the capacity of $B_{\rho}$ with respect to $g$, and
\begin{align}
|c(\rho) - \rho| &\leq k_1 \rho^{1-\tau} \label{eqn_c_r1}\\
|W_{\rho}(x)| &\leq \frac{k_1 \rho}{|x|^{1+\tau}} \label{W}\\
|\partial_iW_{\rho}(x)| &\leq \frac{k_1 \rho}{|x|^{2+\tau}}\label{dW}\\
|\partial_i \partial_j W_{\rho}(x)| &\leq \frac{k_1 \rho}{|x|^{3+\tau}}.\label{ddW}
\end{align}
Moreover,
$$\left|\Hess(\varphi_\rho)\left(\frac{\nabla \varphi_\rho}{|\nabla \varphi_\rho|},\frac{\nabla \varphi_\rho}{|\nabla \varphi_\rho|} \right) +\frac{2c(\rho)^3}{|x|^7}\right| \leq \frac{k_1\rho^3}{|x|^{7+\tau}}$$
\end{lemma}
\begin{proof}
We will use \cite[Lemma A.2]{MMT}, which establishes the expansion for the potential of a fixed set, and then extend its proof as needed. For convenience here, we will tend to work with the opposite boundary conditions on harmonic functions, i.e.  values of 1 on the inner boundary and 0 at infinity.

Let $K$ be such that the AF coordinate chart $(x^1, x^2, x^3)$ is defined on $M \setminus K$, and without loss of generality assume the chart is defined for all $|x| \geq 1$. Let $\Phi$ be $g$-harmonic outside $B_1 \supset K$, equaling 1 on $\partial B_1$ and approaching 0 at infinity.
We have by \cite[Lemma A.2]{MMT}
\begin{equation}
\label{eqn_phi}
\Phi(x) = \frac{c_1}{|x|} + E(x), 
\end{equation}
for a constant $c_1>0$ (equaling the capacity of $B_1(0)$), where the error term satisfies
$$|E(x)| \leq k_0|x|^{-1-\tau} $$
for some $k_0>0$.

Fix any $\epsilon \in (\tau, 1)$. For a parameter $a>0$, define
$$\psi_a(x) = a\Phi(x) - \frac{1}{|x|^{1+\epsilon}},$$
a smooth function on $M \setminus B_1$. Let $\Delta$ and $\Delta_0$ be the $g$- and Euclidean Laplacians. Since $\Delta \Phi=0$, we have:
\begin{align*}
\Delta \psi_a &= \Delta_0 \left(- \frac{1}{|x|^{1+\epsilon}}\right) + (\Delta - \Delta_0) \left(- \frac{1}{|x|^{1+\epsilon}}\right)\\
&= -\epsilon(1+\epsilon) |x|^{-3-\epsilon} + O(|x|^{-3-\tau-\epsilon})
\end{align*}
where $O(|x|^{-3-\tau-\epsilon})$ is independent of $a$. In particular, for $\rho_0 \geq 1$ sufficiently large and independent of $a$, $\Delta_g \psi_a  \leq 0$ on $M \setminus B_{\rho_0}$. If necessary, increase $\rho_0$ so that
$$1 - \frac{k_0}{c_1 \rho_0^\tau} \geq \frac{1}{2}.$$

For any $\rho \geq \rho_0$, let $\varphi_{\rho}$ be the capacitary potential for $ B_{\rho}$, and let $\tilde \varphi_{\rho} = 1 - \varphi_{\rho}$, which is also harmonic, with the boundary values of 1 on $S_\rho$ and $0$ at infinity.  Choose the parameter
$$a =  \frac{\rho + \rho^{-\epsilon}}{c_1 - k_0 \rho^{-\tau}} >0.$$
In particular for any $x \in \partial B_{\rho}$, by \eqref{eqn_phi},
\begin{align*}
\psi_a(x) &\geq a \left( \frac{c_1}{\rho} -\frac{k_0}{\rho^{1+\tau}}\right) - \rho^{-1-\epsilon} = 1
\end{align*}
by the choice of $a$. Then by the maximum principle, $\psi_a \geq \tilde \varphi_{\rho}$ on $M \setminus \mathring B_{\rho}$. Then again by \eqref{eqn_phi}, we have for $x \in M \setminus B_\rho$
\begin{align}
\tilde \varphi_{\rho}(x) &\leq \frac{\rho + \rho^{-\epsilon}}{c_1 - k_0 \rho^{-\tau}}  \left( \frac{c_1}{|x|} + \frac{k_0}{|x|^{1+\tau}}\right)\nonumber\\
&\leq  \frac{\rho + O(\rho^{1-\tau})}{|x|} + \frac{2k_0\rho}{c_1|x|^{1+\tau}} \label{phi_upper},
\end{align}
having used $\epsilon > \tau$.
A similar argument shows
\begin{align}
\tilde \varphi_{\rho}(x) &\geq \frac{\rho - O(\rho^{1-\tau})}{|x|} - \frac{2k_0\rho}{c_1|x|^{1+\tau}} \label{phi_lower}.
\end{align}
The novelty here is first that we have uniform control on the $1/|x|$ coefficient in terms of $\rho$ and second that we have uniform control on the remainder term in terms of $\rho$. As explained in \cite[Lemma A.2]{MMT}, interior Schauder estimates \cite[Theorem 6.2]{GT} imply that there exists a constant $C>0$ independent of $\rho$ such that
$$ |\partial \tilde \varphi_{\rho}| \leq \frac{C\rho}{|x|^2}, \qquad |\partial^2\tilde \varphi_{\rho}| \leq \frac{C\rho}{|x|^3}$$
on $M \setminus B_{\rho}$, where again the point is that we have control in terms of $\rho$. The same argument in  \cite[Lemma A.2]{MMT}, but accounting for this factor of $\rho$, shows that
$$\tilde \varphi_{\rho} = \frac{c(\rho)}{|x|} + W_{\rho}(x)$$
for a function $c$ depending on $\rho$ only, where
$$|W_{\rho}(x)| \leq \frac{C\rho}{|x|^{1+\tau}}.$$
This expression for $\tilde \varphi_{\rho}$, along with \eqref{phi_upper} and \eqref{phi_lower}, implies \eqref{eqn_c_r1}. (Of course, $\varphi_{\rho} = 1 - \tilde \varphi_{\rho}$.)

Interior Schauder estimates give \eqref{dW} and \eqref{ddW}, and it is easy to check that $c(\rho)$ equals the capacity of $B_{\rho}$ with respect to $g$.

Last, we address the final claim. By asymptotic flatness we have $\Hess(\varphi_{\rho}) = \Hess_0(\varphi_{\rho}) +  O(\rho r^{-3-\tau})$. Second
$$\nabla \varphi_{\rho} = \nabla_0 \varphi_{\rho} + X_{\rho},$$
where $|X_{\rho}| = O(\rho r^{-2-\tau}).$ In particular,
\begin{align*}
\Hess(\varphi_{\rho})(\nabla \varphi_{\rho}, \nabla \varphi_{\rho}) &= \Hess_0(\varphi_{\rho}) (\nabla_g \varphi_{\rho}, \nabla_g \varphi_{\rho}) + O( \rho^3r^{-7-\tau})\\
&= \Hess_0(\varphi_{\rho}) (\nabla_0 \varphi_{\rho} + X_{\rho}, \nabla_0 \varphi_{\rho} +X_{\rho}) + O( \rho^3r^{-7-\tau})\\
&= \Hess_0(\varphi_{\rho}) (\nabla_0 \varphi_{\rho} , \nabla_0 \varphi_{\rho}) + O( \rho^3r^{-7-\tau})\\
&= \Hess_0(\psi_{\rho}) (\nabla_0 \psi_{\rho} , \nabla_0 \psi_{\rho}) + O( \rho^3r^{-7-\tau}),
\end{align*}
where $\psi_{\rho}= 1 - \frac{c(\rho)}{|x|}$ is the leading part of $\varphi_{\rho}$. It is elementary to check
$$\Hess_0(\psi_{\rho}) (\nabla_0 \psi_{\rho} , \nabla_0 \psi_{\rho})  = \sum_{i,j=1}^3(\psi_{\rho})_{ij} (\psi_{\rho})_{i} (\psi_{\rho})_{j} = \frac{-2c(\rho)^3}{|x|^7},$$
so the result follows.
\end{proof}

\subsection{Capacity expansion for large balls}
As a consequence of results proved so far, we obtain an expansion for the capacity of a large ball in an AF manifold as a function of the radius. 
\begin{cor}
\label{cor_expansion}
Let $(M,g)$ be an asymptotically flat 3-manifold of nonnegative scalar curvature, with boundary either empty or minimal. Let $B_{r}$ be the compact region bounded by the coordinate sphere of radius $r$ in some AF coordinate chart. Then for large $r$,
$$\capac(B_r) = r +\frac{\beta(r)}{2r} - \frac{m}{2} + o(1),$$
where $\beta(r)$ is defined by \eqref{eqn_beta}.
\end{cor}
The proof follows from Theorems \ref{thm_lb} and \ref{thm_ub}, which together show the limit of the difference between the volume radius and the capacity of $B_r$ equals $m$ as $r \to \infty$. The rest then follows from the known expansion for the volume, equation \eqref{eqn_vol_estimate}, coming from \cite{FST}. As pointed out in Remark \ref{rmk_beta}, $\frac{\beta(r)}{2r} = m + o(1)$ in the harmonically flat case.

\subsection{Case of arbitrary exhaustions}
Unfortunately, it is not clear that $m_{CV} \leq m_{ADM}$ holds in general, i.e., allowing for arbitrary exhaustions of $M$. That is, there could a priori be an exhaustion for which the capacity-volume deficit is stricter than for an exhaustion by balls. This includes two main concerns: the optimal exhaustions could a priori feature regions that are not approximately round, and/or the regions could be far ``off center.'' (Recall these concerns will be eliminated if the metric is harmonically flat at infinity; see Theorem \ref{thm_HF}.) In this section we show two results in favor of this inequality for exhaustions that are not so symmetric.

First, we observe a corollary to Theorem \ref{thm_ub} above that still gives an upper bound on $m_{CV}$ in the case that the exhaustion is not quite by balls.
\begin{cor}
\label{cor_r_j}
Let $(M,g)$ be an asymptotically flat 3-manifold with nonnegative scalar curvature, with $\partial M$ is empty or minimal.  Let $\{K_j\}$ be an exhaustion of $M$ by compact sets with
$$\limsup_{j \to \infty}\left( \max r|_{\partial K_j} - \min r|_{\partial K_j}\right) \leq \alpha < \infty,$$
where $r$ is the radial coordinate in an AF coordinate chart. Then 
$$  \limsup_{j \to \infty} \left[ \left(\frac{3|K_j|}{4\pi}\right)^{1/3} - \capac(K_j)  \right] \leq m_{ADM}(M,g) + \alpha.$$
\end{cor}
\begin{proof}
Let $r_j =  \min r|_{\partial K_j}$ and $R_j =  \max r|_{\partial K_j}$, so $B_{r_j} \subseteq K_j \subseteq B_{R_j}$. Then by monotonicity of volume and capacity:
$$|B_{r_j}| \leq |K_j| \leq |B_{R_j}|$$
and
$$\capac(B_{r_j}) \leq \capac(K_j) \leq \capac(B_{R_j}).$$
Then
\begin{align*}
\left(\frac{3|K_j|}{4\pi}\right)^{1/3} - \capac(K_j)  &\leq \left[\left(\frac{|B_{r_j}|}{4\pi}\right)^{1/3} - \capac(B_{r_j})\right] + \left[\left(\frac{|B_{R_j}|}{4\pi}\right)^{1/3}  -\left(\frac{|B_{r_j}|}{4\pi}\right)^{1/3} \right].
\end{align*}
By Theorem \ref{thm_ub}, $\limsup_{j \to \infty}$ of the first bracketed term is at most $m_{ADM}(M,g)$. Using asymptotic flatness, it is straightforward to show $\limsup_{j \to \infty}$ of the second term is at most $\alpha$ (see \eqref{vol}).
\end{proof}

Second, we show that if an exhaustion is competitive for $m_{CV}$ then it must be becoming more symmetric at infinity in a precise volumetric sense. This will use a quantitative version of the isocapacitary inequality due to Fusco, Maggi, and Pratelli \cite{FMP}.

\begin{definition}
The \emph{Fraenkel asymmetry} of a compact set $\Omega \subset \R^3$ is defined to be
$$\mathcal{A}(\Omega) = \inf_B \left\{\frac{|\Omega \triangle B|}{|B|} \; : \; B \text{ is a closed ball with the same volume as } \Omega \right\}.$$
\end{definition}
Now, consider an asymptotically flat 3-manifold $(M,g)$ with a fixed asymptotically flat coordinate chart on $M \setminus K_0$. If $K$ is any compact set containing $K_0$, we may definite $\mathcal{A}(K)$ in a natural way using the Euclidean metric in the coordinate chart (this is clarified in the proof below).

\begin{prop}
Let $\{K_j\}$ be an exhaustion of an AF 3-manifold $(M,g)$ by compact sets. Let $\mathcal{A}(K_j)$ be the sequence of corresponding Fraenkel asymmetries in a fixed AF coordinate chart, well-defined for $j$ sufficiently large. If
\begin{equation}
\label{eqn_A_K_j}
\liminf_{j \to \infty} \mathcal{A}(K_j) > 0,
\end{equation}
then
$$\limsup_{j \to \infty} \left[ \left(\frac{3|K_j|}{4\pi}\right)^{\frac{1}{3}} - \capac(K_j)\right] = -\infty.$$
In particular, for the purposes of showing ``$m_{CV}(M,g) \leq m_{ADM}(M,g)$'' we may without loss of generality consider only exhaustions for which $\displaystyle\lim_{j \to \infty} \mathcal{A}(K_j) = 0$.
\end{prop}
Note that although $\mathcal{A}(K)$ depends on the chart, the value of \eqref{eqn_A_K_j} does not.

\begin{proof}
First, we note that for uniformly equivalent Riemannian metrics,
$$\Lambda^{-2} g_2 \leq g_1 \leq  \Lambda^2 g_2,$$
on a 3-manifold, we have corresponding inequalities for the volume of a set
$$\Lambda^{-3} |K|_{g_2} \leq  |K|_{g_1} \leq  \Lambda^{3} |K|_{g_2}$$
as well as (if defined, e.g. for AF metrics) the capacity:
$$\Lambda^{-3} \capac_{g_2}(K) \leq  \capac_{g_1}(K) \leq  \Lambda^{3} \capac_{g_2}(K),$$
which can be seen from the co-area formula.

Now let $(M,g)$ be an AF 3-manifold. By Lemma \ref{lemma_end}, for the purposes of evaluating $m_{CV}$ we may without loss of generality assume that $M$ is diffeomorphic to $\R^3$ and that the standard $(x^1, x^2, x^3)$ coordinate system is an AF coordinate system for $g$. In particular, the Euclidean metric $\delta_{ij}$ is well-defined on $M$.

Let $\{K_j\}$ be an exhaustion of $M$ by compact sets. By asymptotic flatness we may find a sequence of positive real numbers $\epsilon_j$ converging to zero such that
$$(1+\epsilon_j)^{-2} \delta \leq g \leq (1+\epsilon_j)^2 \delta \qquad \text{ on }M \setminus K_j.$$
Let $v_j = \left(\frac{3|K_j|_0}{4\pi}\right)^{1/3}$ be the Euclidean volume radius of $K_j$. By the quantitative isocapactiary inequality given in \cite[Theorem 1.2]{FMP}
$$\frac{\capac_0(K_j)}{v_j} \geq 1 + c \mathcal{A}(K_j)^4$$
for a universal constant $c$.
Then:
\begin{align*}
 \left(\frac{3|K_j|_g}{4\pi}\right)^{\frac{1}{3}} - \capac_g(K_j) &\leq  (1+\epsilon_j)v_j - (1+\epsilon_j)^{-3}\capac_
0(K_j)\\
 &\leq v_j \left[1+\epsilon_j - (1+\epsilon_j)^{-3}(1+\mathcal{A}(K_j)^4) \right]
\end{align*}
Now, if $\displaystyle \liminf_{j \to \infty} \mathcal{A}(K_j) > 0$, the $\limsup$ of the above is $-\infty$, since $\epsilon_j \to 0$ and $v_j \to \infty$.
\end{proof}
Unfortunately, this argument does not provide insight into the centering of the balls that minimize for $K_j$ in the definition of Fraenkel asymmetry when $K_j$ is an optimal sequence for $m_{CV}$.

\section{Upper bound on $m_{CV}$ under  harmonically flat asymptotics}
\label{sec_HF}
In this section we prove Theorem \ref{thm_HF}, i.e. $m_{CV} \leq m_{ADM}$ for metrics with nice asymptotics.
\begin{definition}
\label{def_HF}
An asymptotically flat Riemannian $3$-manifold $(M,g)$ is \emph{harmonically flat at infinity (HF)} if there exists a ``harmonically flat coordinate system'' $(x^1, x^2, x^3)$ on $M \setminus K$, i.e.
$$g_{ij} = U^{4} \delta_{ij},$$
on $M \setminus K$ for some function $U>0$, where $\Delta_0 U = 0$ and $U(x) \to 1$ as $|x| \to \infty$. (Here $\Delta_0$ is the Euclidean Laplacian on $\R^3$.)
\end{definition}
Note the Schwarzschild manifold is HF.

It is well known that the harmonic function $U$ appearing in Definition \ref{def_HF} admits an expansion at infinity of the form:
\begin{equation}
\label{eqn_U}
U(x) = 1 + \frac{m}{2|x|} + O (|x|^{-2}),
\end{equation}
where $m$ is the ADM mass of $g$. The fact that $\Delta_0 U=0$ implies that $g$ as above has zero scalar curvature outside of $K$. 

\begin{proof}[Proof of Theorem \ref{thm_HF}]
By Lemma \ref{lemma_end}, $m_{CV}(M,g)$ only depends on the geometry outside of a large compact set. Thus, we may assume without loss of generality that $M$ is diffeomorphic to $\R^3$ and use standard coordinates $(	x^1,x^2, x^3)$. Moreover, we may assume there exists a compact set $\Omega \subset \R^3$ containing the coordinate ball of radius 1, such that on $M \setminus \Omega$ we have
$$g_{ij} = U^{4} \delta_{ij}$$
for a $\delta$-harmonic function $U>0$ on $M \setminus \Omega$ that approaches 1 at infinity.

Let $K$ be any compact set that contains $\Omega$, with $\partial K$ smooth. The strategy of the proof is to compare the capacity and volume of $K$ each with respect to $\delta$ and $g$, then use the isocapacitary inequality.

For the capacity, we let $\varphi_0$ be the capacitary potential for $K$ with respect to $\delta$. Since $\varphi_0$ and $U>0$ are $\delta$-harmonic, the function $\varphi:=\varphi_0 U^{-1}$  on $M \setminus \mathring K$ is harmonic with respect to $g$, and moreover has the same boundary values as $\varphi_0$. In particular, $\varphi$ is the capacitary potential of $K$ in $(M,g)$. From this it is straightforward to show
\begin{equation}
\label{cap_g_delta}
\capac_g(K) = \frac{1}{4\pi} \lim_{r \to \infty}\int_{S_r} \nu(\varphi)dA = \capac_{0}(K) + \frac{m}{2}.
\end{equation}

Next, for the volume, we impose the assumption that the ADM mass $m$ is nonnegative. Let $B_{\rho_0}$ be a coordinate ball in $M \cong \R^3$ of radius $\rho_0$ sufficiently large so that $B_{\rho_0} \supset \Omega$. Let $K$ be any compact set that contains $B_{\rho_0}$, of sufficiently large $\delta$-volume so that the volume radius of $K$, i.e., $\rho_k= \left(\frac{3|K|_\delta}{4\pi}\right)^{1/3}$, is greater than $\rho_0$.

We proceed to find an upper bound on the $g$-volume of $K$ in terms of the Euclidean volume and the ADM mass.
We have (using $dx$ to denote $dx^1 dx^2 dx^3$)
\begin{align}
|K|_g &= |\Omega|_g + \int_{K \setminus \Omega} \sqrt{\det(g_{ij})} dx\nonumber\\
&=  |\Omega|_g + \int_{K \setminus \Omega} \left(1+ \frac{3m}{|x|} + W\right) dx\nonumber\\
&=  |K_\delta| + |\Omega|_g- |\Omega|_\delta  + \int_{K \setminus \Omega} \left(\frac{3m}{|x|} + W\right) dx, \label{eqn_K_g}
\end{align}
where $W$ is a continuous function on $M \setminus \Omega$ satisfying $|W(x)| \leq \frac{C}{|x|^2}$ for a constant $C>0$ depending only on $g$ and $\Omega$. We next estimate the integral involving $W$.

Note the coordinate ball $B_{\rho_k}$ about the origin contains $B_{\rho_0}$. Now,

\begin{align*}
\int_{K \setminus \Omega} W  dx &\leq C \int_{K \setminus B_{\rho_0}} \frac{1}{|x|^2} dx + C  \int_{B_{\rho_0} \setminus \Omega} \frac{1}{|x|^2} dx\\
&\leq \frac{C}{\rho_0} \int_{K \setminus B_{\rho_0}} \frac{1}{|x|} dx + C  |B_{\rho_0}|_\delta\\
&\leq \frac{C}{\rho_0} \int_{K \setminus \Omega} \frac{1}{|x|} dx + C  |B_{\rho_0}|_\delta
\end{align*}
since $\Omega \supset B_1$. Continuing \eqref{eqn_K_g} we have
\begin{align*}
|K|_g &\leq |K|_\delta + |\Omega|_g - |\Omega|_\delta + C|B_{\rho_0}| + \int_{K \setminus \Omega} \left(\frac{3m + C/\rho_0}{|x|} \right) dx.
\end{align*}
For this integral, we take advantage of the fact that the integrand is rotationally symmetric and decreasing as a function of the radius, as $m \geq 0$. In particular, since $B_{\rho_k}$ by construction has the same Euclidean volume as $K$, and since $B_{\rho_k}$ and $K$ contain $\Omega$, we claim
$$\int_{K \setminus \Omega} \left(\frac{3m + C/\rho_0}{|x|} \right) dx \leq \int_{B_{\rho_k} \setminus \Omega} \left(\frac{3m + C/\rho_0}{|x|} \right) dx.$$
This can be verified readily; let $f(r)$ stand for the integrand. Since $|K|_\delta = |B_{\rho_k}|_\delta$, we have $|K \setminus B_{\rho_k}|_\delta = |B_{\rho_k} \setminus K|_\delta$. Then
\begin{align*}
\int_{ K \setminus \Omega} f(r) dx &= \int_{ (K \cap B_{\rho_k}) \setminus \Omega} f(r) dx + \int_{ K \setminus B_{\rho_k}} f(r) dx\\
&\leq \int_{ (K \cap B_{\rho_k}) \setminus \Omega} f(r) dx + f(\rho_k)\int_{ K \setminus B_{\rho_k}}  dx\\
&= \int_{ (K \cap B_{\rho_k}) \setminus \Omega} f(r) dx + f(\rho_k)|B_{\rho_k} \setminus K|\\
&\leq \int_{ (K \cap B_{\rho_k}) \setminus \Omega} f(r) dx + \int_{B_{\rho_k} \setminus K} f(r) dx,
\end{align*}
which proves the claim.

Continuing on,
\begin{align*}
|K|_g &\leq |K|_\delta + |\Omega|_g - |\Omega|_\delta + C|B_{\rho_0}| + \int_{B_{\rho_k} \setminus B_{\rho_0}} \left(\frac{3m + C/\rho_0}{|x|} \right) dx + \int_{B_{\rho_0} \setminus \Omega} \left(\frac{3m + C/\rho_0}{|x|} \right) dx.
\end{align*}
The first integral is easily evaluated; for the second we can bound above by $(3m + C/\rho_0) |B_{\rho_0}|$. Doing so (and using the definition of $\rho_k$), we obtain:
\begin{align*}
|K|_g &\leq |K|_\delta + |\Omega|_g - |\Omega|_\delta + (C +3m + C/\rho_0)|B_{\rho_0}| + 2\pi (3m + C/\rho_0) \left(\frac{3|K|_{\delta}}{4\pi}\right)^{2/3}.
\end{align*}
Then for a constant $\Lambda$ independent of $K$ (depending on $C$, $m$, $\Omega$, and $\rho_0$), we have: 
\begin{align*}
|K|_g &\leq |K|_\delta \left[1 + \Lambda |K_{\delta}|^{-1} + 2\pi (3m + C/\rho_0) \left(\frac{3}{4\pi}\right)^{2/3} |K|_{\delta}^{-1/3}\right],
\end{align*}
so
\begin{align*}
|K|_g^{1/3} &\leq |K|_\delta^{1/3} \left[1  + \frac{2\pi}{3} (3m + C/\rho_0) \left(\frac{3}{4\pi}\right)^{2/3} |K|_{\delta}^{-1/3} + O(|K|_\delta^{-2/3})\right].
\end{align*}
In particular, we may estimate the volume radius of $K$ with respect to $g$:
\begin{align*}
\left(\frac{3|K|_g}{4\pi}\right)^{1/3} &\leq \left(\frac{3|K|_\delta}{4\pi}\right)^{1/3}  +  \frac{3m}{2} + \frac{C}{2\rho_0} + O(|K|_\delta^{-1/3}).
\end{align*}
Using this inequality along with \eqref{cap_g_delta} and the Euclidean isocapacitary inequality, we have
\begin{align*}
\left(\frac{3|K|_g}{4\pi} \right)^{1/3} - \capac_g(K) &\leq m + \frac{C}{2\rho_0} + O(|K|_{\delta}^{-1/3}).
\end{align*}

In particular, if $K_{j}$ is an exhaustion of $\R^3$ by compact sets with smooth boundary, then for $j$ sufficiently large, the hypotheses on $K$ are met (that is, $K_j$ contains $B_{\rho_0}$ and $|K_j|_{\delta}$ grows arbitrarily large). Then
$$\limsup_{j \to \infty} \left[\left(\frac{3|K_j |_g}{4\pi} \right)^{1/3} - \capac_g(K_j)\right] \leq m + \frac{C}{2\rho_0}.$$
Since $\rho_0$ can be arbitrarily large, this $\limsup$ is at most $m$. In particular, $m_{CV} \leq m_{ADM}$, having used Lemma \ref{lemma_smooth}(b).
\end{proof}

\section{Capacity-volume mass in lower regularity}
\label{sec_low_reg}
A major source of motivation for studying the isocapacitary mass is the desire to generalize the ADM mass to lower regularity settings. The discussion below includes some indication that capacity behaves more favorably than area in low regularity.

We first recall that the capacity continues to be well-defined even if $(M,g)$ is merely \emph{$C^0$ asymptotically flat}, i.e. if $M$ is smooth, $g$ is a $C^0$ Riemannian metric on $M$, and outside of a compact set we have coordinates $(x^1,x^2,x^3)$ so that
\begin{equation}
\label{eqn_C0_AF}
g_{ij} = \delta_{ij} + O(|x|^{-q})
\end{equation}
for some $q > 0$ \cite{JL1}. In particular, Definition \ref{def_mCV} of $m_{CV}$ continues to make sense for such spaces, just as Huisken's isoperimetric mass does. A natural question is whether $m_{iso}$ and $m_{CV}$ necessarily agree for $C^0$ AF metrics of positive mass, i.e. without assuming smoothness, but we do not pursue this here.

We point out it is not difficult to show that if $M$ is a smooth manifold equipped with a sequence of $C^0$ AF metrics $\{g_i\}$ that converge locally uniformly to a $C^0$ AF metric $g$, then for any fixed compact set $K \subset M$,
\begin{equation}
\label{eqn_USC}
\capac_g(K) \geq \limsup_{i \to \infty} \capac_{g_i}(K),
\end{equation}
and of course $|K|_{g_i} \to |K|_g$. 
In particular, the quasi-local quantity $\left( \frac{3|K|}{4\pi}\right)^{1/3} - \capac(K)$ is lower semicontinuous under such convergence for a fixed $K$. Thus any jump in its value ``goes the right way'' for trying to prove lower semicontinuity of mass as in \cites{J_LSC,JL1, JL2}. We also expect this to continue to hold for weaker notions of convergence, particularly Sormani--Wenger intrinsic flat volume ($\mathcal{VF}$) convergence \cite{SW}. Note that $\mathcal{VF}$ convergence was used in \cite{JL2}, and it required significant work to address the fact that the isoperimetric mass expression
$$\frac{2}{|\partial K|}\left(|K| - \frac{1}{6\sqrt{\pi}}|\partial K|^{3/2}\right).$$
is \emph{not} always well behaved under $\mathcal{VF}$; that is, the perimeter can drop in a limit, which ``goes the wrong way'' for proving lower semicontinuity of mass. In other words, there is reason to believe the capacity-volume mass may be better behaved than the isoperimetric mass under low regularity convergence.

It is also worth pointing out that volume and capacity share a basic property that perimeter lacks: monotonicity under set inclusion. This was useful, for example, in Corollary \ref{cor_r_j}.

Next, we discuss $L^p$ convergence of Riemannian metrics, restricting to limits that are continuous to avoid technical difficulties of working with almost-everywhere defined metrics. We refer the reader to Allen and Sormani's recent paper that includes results and examples relating $L^p$ convergence of metric tensors to Gromov--Hausdorff and Sormani--Wenger intrinsic flat  convergence \cite{AS}.

Fix a $C^0$ AF Riemannian 3-manifold $(M,g)$, and let $\{g_i\}$ be a sequence of $C^0$ AF metrics on $M$. Of course, for a compact set $K \subset M$, the capacity and volume of $K$ are well-defined with respect to $g$ and $g_i$.

\begin{prop}
If $g_i \to g$ in $L^p_{\text{\emph{loc}}}(M)$ for $p \geq 3$ and $g_i \geq g$ 
pointwise, then
$$\lim_{i \to \infty} |K|_{g_i} = |K|_g$$
and
$$\limsup_{i \to \infty} \capac_{g_i}(K) \leq \capac_g(K),$$
i.e. the volume and capacity are continuous and upper semicontinuous, respectively.
\end{prop}
The first claim was proved in \cite[Theorem 4.3]{AS}. We remark that the condition $g_i \geq g$ alone does not imply an inequality between $\capac_{g_i}(K)$ and $\capac_g(K)$.
\begin{proof}
First, we recall how the $L^p$ norm is defined for Riemannian metrics, following \cite{AS}. Let $\lambda_1^2 \geq \lambda_2^2 \geq \lambda_3^2>0$ denote the eigenvalues of $g_i$ with respect to $g$ at a point. Let $|g_i|_g = \sqrt{\lambda_1^2 + \lambda_2^2 + \lambda_3^2}$, and, for a compact set $C \subset M$
$$\|g_i\|_{L^p_g(C)} = \left(\int_C |g_i|_g^p dV_g\right)^{\frac{1}{p}}.$$
In particular, if $g_i \to g$ in $L^p_{\text{loc}}$ for $p \geq 3$, then
$$\|g_i\|_{L^3_g(C)} \to \|g\|_{L^3_g(C)} = \sqrt{3} \left(|C|_g\right)^{\frac{1}{3}}.$$

We recall the proof from \cite[Theorem 4.3]{AS} of the first claim, which follows from the trace-determinant inequality.  Letting $F_i$ be the product of the eigenvalues of $g_i$ with respect to $g$, we have
\begin{align*}
|K|_{g_i} &= \int_K \sqrt{F_i} dV_g\\
&\leq \frac{1}{3^{3/2}} \int_K |g_i|^3 dV_{g}
\end{align*}
which converges to $|K|_g$. The reverse inequality follows from the pointwise condition on $g_i$. 
 
For the second claim, begin with a test function $\phi$ for $\capac_g(K)$ that comes within $\epsilon$ of the infimum. By a cutoff argument, we may assume $d\phi$ has compact support, call it $C$. By the pointwise lower bound on $g_i$, we have the upper bound $ |\nabla \phi|_{g_i} \leq |\nabla \phi|_{g}$. Then:
\begin{align*}
4\pi\capac_{g_i}(K) &\leq \int_M |\nabla \phi|^2_{g_i} dV_{g_i}\\
&= \int_C |\nabla \phi|^2_{g_i} \sqrt{F_i} dV_g\\
&\leq \int_C |\nabla \phi|^2 \frac{1}{3^{3/2}}|g_i|^3 dV_g.
\end{align*}
Since $\phi$ is Lipschitz on $C$ we see this integral converges to $\int_C |\nabla \phi|^2  dV_g$. Thus,
\begin{align*}
\limsup_{i \to \infty} \capac_{g_i}(K) &\leq \frac{1}{4\pi} \int_C |\nabla \phi|^2   dV_g\\
&\leq \capac_g(K) + \epsilon.
\end{align*}
Since $\epsilon$ was arbitrary, the proof is complete.
\end{proof}
By contrast, under the hypotheses in the proposition, the perimeter can ``jump down,'' as simple conformal examples where the $g_i$ are much larger than $g$ on a thin shell around $\partial K$ show. Again, this behavior ``goes the wrong way.''

In future work, we will extend the definition of $m_{CV}$ to a class of ``asymptotically flat'' metric measure spaces that lack Riemannian structure.

\section{Some conjectures}
\label{sec_conj}
The main question left unanswered here is whether $m_{CV} \leq m_{ADM}$ holds in full generality, i.e. under the hypotheses of Theorem \ref{thm_ub}. Below, we discuss one approach to attacking this and propose some conjectures that may have independent interest.

Consider a bounded region $\Omega$ in an asymptotically flat 3-manifold $(M,g)$, with $S= \partial \Omega$ smooth. Then $\Omega$ is a critical point for capacity among nearby regions of the same volume if and only if the capacitary potential $\varphi$ for $\Omega$ has constant normal derivative on $S$:
$$\left.\frac{\partial\varphi}{\partial \nu}\right|_S = \text{const.}$$
Such surfaces arise as solutions to D. Bernoulli's free boundary problem; typically this problem involves one fixed boundary, and the free (interior or exterior) boundary is sought to make the above equation satisfied for a harmonic function with Dirichlet boundary values of 0 and 1. We refer the reader to \cite{FR} for example, which includes many references and several interpretations of such surfaces from classical physics. 

In our case, the fixed boundary is at infinity in the asymptotically flat end. For convenience, we give:
\begin{definition}
Let $(M,g)$ be an asymptotically flat 3-manifold. A surface $S= \partial A$ for a compact set $A$ is a \emph{Bernoulli surface} (relative to infinity) if there exists a harmonic function $u$ on $M \setminus A$, approaching 1 at infinity, equaling zero on $S$, and such that $\frac{\partial u}{\partial \nu}$ is constant on $S$ (where $S$ is the unit normal pointing into $M \setminus A$. 
\end{definition}
Of course, the Bernoulli condition (unlike constant mean curvature, for example) is non-local. Examples of Bernoulli surfaces include spheres in Euclidean space and $r=$ constant coordinate spheres in rotationally symmetric asymptotically flat manifolds.

The following conjectures are inspired by the analogous results for constant mean curvature surfaces and large isoperimetric regions. First, recall that an AF 3-manifold with nonnegative scalar curvature and positive ADM mass admits a unique foliation on the complement of a compact set by stable CMC surfaces. This problem began with Huisken and Yau \cite{HY} and culminated in an optimal result by Nerz \cite{Ner}. We refer the reader to Appendix A of \cite{CESY} and to \cite{Ner} for additional references and precise statements.

\begin{conj}
\label{conj1}
Let $(M,g)$ be an asymptotically flat 3-manifold with nonnegative scalar curvature and positive ADM mass. Then there exists a compact set $K$ such that $M \setminus K$ is foliated by Bernoulli surfaces of spherical topology.
\end{conj}
Some results are already known regarding foliations by Bernoulli surfaces (relative to some fixed surface rather than infinity). We refer the reader to an early paper on this topic by Acker \cite{Ack} and a paper by Henrot and Onodera \cite{HO} for some recent results and many references.

Chodosh, Eichmair, Shi, and Yu showed that for $V>0$ sufficiently large, there is a unique isoperimetric region $\Omega_V$ with volume $V$ in an AF 3-manifold of nonnegative scalar curvature with $m_{ADM}>0$ (whose boundary is empty or else consists of minimal surfaces), and moreover $\partial \Omega_V$ consists of $\partial M$ and a leaf of the canonical foliation by CMC surfaces \cite[Theorem 1.1]{CESY}. By analogy, we conjecture:

\begin{conj}
\label{conj2}
Let $(M,g)$ be an asymptotically flat 3-manifold with nonnegative scalar curvature and positive ADM mass. Assume $\partial M$ is empty or minimal. Then for $V>0$ sufficiently large there exists a unique compact region containing $\partial M$ with volume $V$ whose capacity is minimized among such regions.
\end{conj}
If both of these conjectures were true, it would then be expected that the boundary of the minimizer in Conjecture \ref{conj2} is the union of $\partial M$ and a leaf of the foliation in Conjecture \ref{conj1}. At that point, if the foliation were sufficiently asymptotically round, then the conjectured inequality $m_{CV} \leq m_{ADM}$ should follow, perhaps as in Corollary \ref{cor_r_j}.

\section{Appendix: facts regarding Huisken's isoperimetric mass}

\subsection{The error term in isoperimetric mass estimate}
The constant ``$C$'' appearing in Theorem \ref{thm_JL} was not made explicit in \cite{JL1}, though this is readily done below. Recall that in the proof of Theorem \ref{thm_ub} we required knowledge of the dependence of $C$ on $I_0$.
\begin{lemma}
\label{lemma_JL}
In Theorem \ref{thm_JL}, the constant $C$ can be chosen to be $\gamma_0 \mu_0^2(1 + I_0^2 c_0^{-3})$, where $\gamma_0$ is a universal constant.
\end{lemma}
\begin{proof}
For $m \geq 0$, let $\phi_m(A):[16\pi m^2,\infty) \to [0,\infty)$ denote the volume bounded between the horizon and the rotationally symmetric sphere of area $A$ in the Schwarzschild manifold of mass $m$. (Note the area of the horizon is $16\pi m^2$.) Huisken's definition of isoperimetric mass was motivated by his observation that
$$\lim_{A \to \infty} \frac{2}{A} \left(\phi_m(A) - \frac{A^{3/2}}{6\sqrt{\pi}}\right) = m.$$
A proof of this appears in \cite[Lemma 10]{JL1}, where it is also observed that the limit is uniform among different values of $m$, given an upper bound on such $m$. An examination of this proof makes it clear that there exists a constant $\gamma$ independent of $m$ and $A$ such that
\begin{equation}
\label{phi_m}
 \frac{2}{A} \left(\phi_m(A) - \frac{A^{3/2}}{6\sqrt{\pi}}\right) \leq m + \frac{\gamma m^2}{\sqrt{A}}.
 \end{equation}

Now in \cite[eqn. (13)]{JL1}, the first inequality below is established for $\Omega$ satisfying the hypotheses of Theorem \ref{thm_JL}; we extend it using \eqref{phi_m}, as well as $m \leq \mu_0$, and  $|\partial \Omega| \geq 36\pi \mu_0^2$. Here $m$ is the ADM mass of $(M,g)$.
\begin{align*}
\frac{2}{|\partial \Omega|} \left( |\Omega| - \frac{|\partial \Omega|^{3/2}}{6\sqrt{\pi}}\right) &\leq \frac{2}{|\partial \Omega|}\left( \phi_m(|\partial \Omega|)-\frac{|\partial \Omega|^{3/2}}{6\sqrt{\pi}} + c_0^{-3} (36\pi \mu_0^2)^{3/2} I_0^2\right)\\
&\leq m + \frac{\gamma \mu_0^2}{\sqrt{|\partial \Omega|}} + \frac{ 2c_0^{-3} (36\pi \mu_0^2) I_0^2}{\sqrt{|\partial \Omega|}}
\end{align*}
which implies the claim. (Note the isoperimetric ratio in \cite{JL1} does not include the normalizing factor $\frac{1}{6\sqrt{\pi}}$ but this only affects the universal constant.)
\end{proof}

\subsection{Alternative formula for isoperimetric mass}
One might wonder whether Huisken's formula for the isoperimetric mass could be rewritten using the difference between the volume radius and the area radius, similar to expression \eqref{m_CV2} for $m_{CV}$. We briefly confirm this here.

\begin{prop}
\label{prop_alt}
Let $(M,g)$ be a $C^0$ asymptotically flat manifold (as in \eqref{eqn_C0_AF}).  Then:
\begin{equation}
\label{eqn_miso*}
m_{iso}(M,g) = 2 \sup_{\{K_j\}} \limsup_{j \to \infty} \left[\left(\frac{3|K_j|_g}{4\pi}\right)^{1/3} - \left(\frac{|\partial^*K_j|_g}{4\pi}\right)^{1/2} \right].
\end{equation}
\end{prop}
\begin{proof}
The proof is very similar to that of Lemma \ref{lemma_mcv}; one shows that in the definition of $m_{iso}$, one may restrict to exhaustions $\{K_j\}$ such that
$$\lim_{j \to \infty} \frac{|\partial^* K_j|^{3/2}}{|K_j|} = 6\sqrt{\pi}.$$
This follows from the $C^0$ asymptotic flatness and the Euclidean isoperimetric inequality, as well as the argument in \cite[Lemma 16]{JL1} to show that a suboptimal  competing exhaustion can be improved by replacement with large balls. At that point, the term appearing in the definition of $m_{iso}$ is factored as a difference of cubes, and the result follows.
\end{proof}

The factor of 2 appearing in \eqref{eqn_miso*} but not in the definition of $m_{CV}$ can be explained as follows. In the Schwarzschild space of mass $m$, for a coordinate ball of radius $r$, the volume radius is $r + \frac{3}{2} m + O(r^{-1})$, the area radius is $r+ m + O(r^{-1})$, and the capacity is $r + \frac{m}{2} + O(r^{-1})$.

\subsection{Higher dimensional formula}
Here we collect the analogous definitions for $m_{iso}$ and $m_{CV}$ in higher dimensions. The ADM mass of an asymptotically flat $n$-manifold, $n \geq 3$ (see \cite{BL} for example) is given by
$$m_{ADM}(M,g) = \lim_{r \to \infty} \frac{1}{2(n-1) \omega_{n-1}}\int_{S_r} \left(\partial_i g_{ij} - \partial_j g_{ii} \right) \frac{x^j}{|x|} dA,$$
where $\omega_{n-1}$ is the hypersurface area of the unit $(n-1)$-sphere, and the Schwarzschild metric of mass $m$ is given by
$$g_{ij} = \left(1+ \frac{m}{2|x|^{n-2}}\right)^{\frac{4}{n-2}} \delta_{ij}.$$
Capacity is defined as in \eqref{eqn_cap}, except with the $4\pi$ replaced with $(n-2)\omega_{n-1}$. 
Direct calculation, similar to Proposition \ref{prop_schwarz}, gives the volume, area, and capacity of coordinate balls in Schwarzschild:
\begin{align*}
V(r) &= \beta_n r^n + \frac{n \omega_{n-1} }{2(n-2)} mr^2  + O(r),\\
A(r) &= \omega_{n-1} r^{n-1} + \frac{(n-1)\omega_{n-1}}{n-2} mr + O(1),\\
\capac(B_r) &= r^{n-2} + \frac{m}{2},
\end{align*}
where $\beta_n = \frac{\omega_{n-1}}{n}$ is the volume of the unit $n$-ball. Using these, the natural generalization of Huisken's definition to dimension $n \geq 3$ is seen to be:
\begin{equation}
m_{iso}(M,g) = \sup_{\{K_j\}} \limsup_{j \to \infty} \frac{2}{\omega_{n-1}} \left(\frac{\omega_{n-1}}{|\partial^* K_j|}\right)^{\frac{2}{n-1}} \left(|K_j| -\frac{1}{n (\omega_{n-1})^\frac{1}{n-1}} |\partial^* K_j|^{\frac{n}{n-1}}\right).
\end{equation}
(Note that the difference in the volume and area radii (as in \eqref{eqn_miso*}) for Schwarzschild is $O(r^{3-n})$, so that expression does not limit to a nonzero multiple of the mass if $n >3$.)

Similar calculations motivate the following definition for the capacity-volume mass in dimension $n \geq 3$:
\begin{equation}
m_{CV}(M,g) = \frac{2(n-2)}{(n-1) \omega_{n-1}\capac(K_j)^{\frac{2}{n-2}}}  \left( |K_j| - \beta_n \capac(K_j)^{\frac{n}{n-2}}\right).
\end{equation}

From Bray's proof that the coordinate balls in Schwarzschild manifolds of positive mass are isoperimetric regions \cite{bray_thesis} and the higher-dimensional generalization due to Bray and Morgan \cite{bray_morgan}, it follows that $m_{iso} = m_{ADM}$ for $m>0$ Schwarzschild. Using the Bray--Morgan result in the proof of the isocapacitary inequality for Schwarzschild by Mantoulidis, Miao, and Tam \cite[Theorem 1.10]{MMT} (which was for $n=3$), we also have $m_{CV} = m_{ADM}$ for $m>0$ Schwarzschild in dimension $n \geq 3$.

\end{document}